\documentclass{article}%
\usepackage{amsmath}
\usepackage{amsthm}
\usepackage{graphicx}
\usepackage{subfigure}
\usepackage{fullpage}
\usepackage{color}
\usepackage{listings}
\usepackage[colorlinks=true, linkcolor=blue]{hyperref}%
\usepackage{amsfonts}%
\usepackage{amssymb}
\theoremstyle{plain}
\newtheorem{theorem}{Theorem}[section]
\newtheorem{lemma}[theorem]{Lemma}
\newtheorem{corollary}[theorem]{Corollary}
\newtheorem{proposition}[theorem]{Proposition}
\theoremstyle{definition}
\newtheorem{example}[theorem]{Example}
\newtheorem{definition}[theorem]{Definition}

\newtheorem{claim}[theorem]{Claim}
\newtheorem{problem}[theorem]{Problem}
\newcommand\theoremref[1]{\hyperref[#1]{Theorem~\ref*{#1}}}
\newcommand\lemmaref[1]{\hyperref[#1]{Lemma~\ref*{#1}}}
\newcommand\corollaryref[1]{\hyperref[#1]{Corollary~\ref*{#1}}}
\newcommand\propositionref[1]{\hyperref[#1]{Proposition~\ref*{#1}}}
\newcommand\definitionref[1]{\hyperref[#1]{Definition~\ref*{#1}}}
\newcommand\problemref[1]{\hyperref[#1]{Problem~\ref*{#1}}}
\newcommand{\R}{\mathbf{R}}
\newcommand{\I}{\mathcal{I}}
\newcommand{\V}{\mathcal{V}}
\newcommand{\F}{\mathcal{F}}

\begin{document}

\title{A congruence problem for polyhedra}
\author{Alexander Borisov, Mark Dickinson, Stuart Hastings}
\maketitle

\begin{abstract}
It is well known that to determine a triangle up to congruence  requires three
measurements: three sides, two sides and the included  angle, or one side and
two angles. We consider various  generalizations of this fact to two and three
dimensions. In  particular we consider the following question: given a convex
polyhedron~$P$, how many measurements are required to determine~$P$  up to congruence?

We show that in general the answer is that the number of  measurements
required is equal to the number of edges of the  polyhedron. However, for many
polyhedra fewer measurements suffice;  in the case of the unit cube we show
that nine carefully chosen  measurements are enough.

We also prove a number of analogous results for planar polygons. In
particular we describe a variety of quadrilaterals, including all  rhombi and
all rectangles, that can be determined up to congruence  with only four
measurements, and we prove the existence of $n$-gons  requiring only~$n$
measurements. Finally, we show that one cannot  do better: for any sequence
of~$n$ distinct points in the plane  one needs at least~$n$ measurements to
determine it up to  congruence.

\end{abstract}

\section{Introduction}

We discuss a class of problems about the congruence or similarity of three
dimensional polyhedra. The basic format is the following:

\begin{problem}
\label{motivating_problem}  \label{M}Given two polyhedra in $\mathbf{R}^{3}$
which have the same  combinatorial structure (e.g. both are hexahedra with
four-sided  faces), determine whether a given set of measurements is
sufficient  to ensure that the polyhedra are congruent or similar.
\end{problem}

We will make this more specific by specifying what sorts of
measurements will be allowed. For example, in much of the paper,
allowed measurements will include distances between pairs of vertices,
angles between edges, angles between two intersecting face diagonals
(possibly on different faces with a common vertex) or between a face
diagonal and an edge, and dihedral angles (that is, angles between two
adjoining faces). One motivation for these choices is given below. Sometimes we are more restrictive, for example, allowing only distance measurements. 

In two dimensions this was a fundamental question answered by Euclidean
geometers, as (we hope) every student who takes geometry in high school
learns. If the lengths of the corresponding sides of two triangles are equal,
then the triangles are congruent. The SAS, ASA, and AAS theorems are equally
well known. The extension to other shapes is not often discussed, but we will
have some remarks about the planar problem as well. It is surprising to us
that beyond the famous theorem of Cauchy discussed below, we have been unable
to find much discussion of the problems we consider in the literature, though
we think it is almost certain that they have been studied in the past. We
would be appreciative if any reader can point us to relevant results.

Our approach will usually be to look at the problem locally. If the two
polyhedra are almost congruent, and agree in a certain set of measurements,
are they congruent? At first glance this looks like a basic question in what
is known as rigidity theory, but a little thought shows that it is different.
In rigidity theory, attention is paid to relative positions of vertices,
viewing these as connected by inextensible rods which are hinged at their ends
and so can rotate relative to each other, subject to constraints imposed by
the overall structure of rods. In our problem there is the additional
constraint that in any movement of the vertices, the combinatorial structure
of the polyhedron cannot change. In particular, any vertices that were
coplanar before the movement must be coplanar after the movement. This feature
seems to us to produce an interesting area of study.

Our original motivation for considering this problem came from a very
practical question encountered by one of us (SPH). If one attempts to make
solid wooden models of interesting polyhedra, using standard woodworking
equipment, it is natural to want to check how accurate these models
are.\footnote{The usual method of constructing a  polyhedron is by folding a
paper shell.} As a mathematician one may be attracted first to the Platonic
solids, and of these, the simplest to make appears to be the cube. (The
regular tetrahedron looks harder, because non-right angles seem harder to cut
accurately. )

It is possible to purchase lengths of wood with a square cross section, called
``turning squares'' because they are mostly used in lathes. To make a cube,
all one has to do is use a saw to cut off the correct length piece from a
turning square. Of course, one has to do so in a plane perpendicular to the
planes of sides of the turning square. It is obvious that there are several
degrees of freedom, meaning several ways to go wrong. The piece cut off could
be the wrong length, or you could cut at the wrong angle, or perhaps the cross
section wasn't square to begin with. So, you start measuring to see how well
you have done.

In this measurement, though, it seems reasonable to make some assumptions. The
basic one of interest here is that the saw cuts off a planar slice. You also
assume that this was true at the sawmill where the turning square was made. So
you assume that you have a hexahedron -- a polyhedron with six faces, all of
which are quadrilaterals. Do you have a cube? At this point you are not asking
a question addressed by standard rigidity theory.

One's first impulse may be to measure all of the edges of the
hexahedron, with the thought that if these are equal, then it is
indeed a cube. This is quickly seen to be false, because the faces
could be rhombi. Another intriguing possibility that we considered
early on is that measuring the twelve face diagonals might suffice.
However, we found some examples showing that this was not the case,
and David Allwright \cite{allwright} gave a short and elegant
quaternion-based classification of all hexahedra with equal face
diagonals.  See also an earlier discussion of this problem in \cite{TarMak}.
\footnote{The face diagonals of a hexahedron are equal when these 
diagonals form two congruent regular tetrahedra whose edges intersect in   
pairs. As it turns out, this arrangement is not unique. 
We have included Allwright's analysis of the possibilities 
as Appendix A in the posting of this paper to www.arXiv.org.} 
 Clearly some 
other configuration of measurements, perhaps
including angles, is necessary. It did not take long to come up with
several sets of 12 measurements which did the job, but a proof that
this number was necessary eluded us.

In our experience most people, even most mathematicians, who are presented
with this problem do not see an answer immediately. Apparently the cube is
harder than it looks, and so one would like a simpler problem to get some
clues. The (regular) tetrahedron comes to mind, and so one asks how many
measurements are required to establish that a polyhedron with four triangular
faces is a tetrahedron.

Now we turn to Cauchy's theorem.

\begin{theorem}
[Cauchy, 1839]\label{cauchy}  Two convex polyhedra with corresponding
congruent and similarly situated  faces have equal corresponding dihedral angles.
\end{theorem}

If we measure the six edges of our triangular faced object, and find them
equal, then we have established congruence of the faces of our object to those
of a tetrahedron. Cauchy's theorem tells us that the dihedral angles are the
same and this implies the desired congruence.

For a tetrahedron this result is fairly obvious, but for the two other
Platonic solids with triangular faces, namely the octahedron and the
icosahedron, it is less so. Hence Cauchy's theorem is of practical value to
the (extremely finicky) woodworker, and shows that for these objects, the
number of edges is at least an upper bound on the number of measurements
necessary to prove congruence. From now on we will denote the number of edges
of our polyhedron by $E$, the number of vertices by $V$ and the number of
faces by $F$. We will only consider simply connected polyhedra, so that
Euler's formula, $V+F=E+2,$ holds.

It is not hard to give an example showing the necessity of convexity in
Cauchy's result, but it is one where the two polyhedra being compared are, in
some sense, far apart. It was not easy to determine if convexity was necessary
for local congruence. Can a nonconvex polyhedron with all faces triangular be
perturbed smoothly through a family of noncongruent polyhedra while keeping
the lengths of all edges constant? The answer is yes, as was proved in a
famous and important paper by R. Connelly in 1979.\cite{connelly}

Cauchy's result also gives us an upper bound for the number of
measurements necessary to determine a unit cube: triangulate the cube
by dividing each square face into a pair of triangles.  Then we have a
triangular-faced object with eighteen edges, and by Cauchy's theorem
those eighteen edge measurements suffice to determine the cube up to
congruence.

However, we can do better.  Start by considering a square. We can
approach this algebraically by assuming that one vertex of the square
is at $\left( 0,0\right) $ in the plane.  Without loss of generality
we can also take one edge along the $x$-axis, going from $\left(
  0,0\right) $ to $\left( x_{1},0\right) $ for some $x_{1}>0.$ The
remaining vertices are then $\left( x_{2},x_{3}\right) $ and $\left(
  x_{4,}x_{5}\right) $ and this leads to the conclusion that to
determine five unknowns, we need five equations, and so five
measurements. For example, we could measure the four sides of the
square and one vertex angle, or we could measure a diagonal instead of
the angle.

We then use this to study the cube. Five measurements show that one face is a
square of a specific size. Only four more are needed to specify an adjacent
face, because of the common edge, and the three more for one of the faces
adjoining the first two. The requirement that this is a hexahedron then
implies that we have determined the cube completely, with twelve measurements.
This is a satisfying result because it shows that $E$ measurements suffice for
a cube as well as for the triangular faced Platonic solids.

However, as remarked earlier, at this stage we have not proved the necessity
of twelve measurements, only the sufficiency. One of the most surprising
developments for us in this work was that in fact, twelve are not necessary.
It is possible to determine a cube (including its size) with nine measurements
of distances and face angles. The reason, even more surprisingly, is that only
four measurements are needed to determine the congruence of a quadrilateral to
a specific square, rather than five as seemed so obvious in the argument above.

We will give the algorithm that determines a square in four measurements in
the final section of the paper, which contains a number of remarks about
congruence of polygons. For now, we proceed with developing a general method
for polyhedra. This method will also handle similarity problems, where the
shape of the polyhedron is specified up to a scale change. In determining
similarity, only angle measurements are involved. As the reader might expect,
in general $E-1$ measurements suffice, with one additional length required to
get congruence.

\section{$E$ measurements suffice}

In this section we prove that for a \emph{convex} polyhedron~$P$ with~$E$
edges, there is a set of~$E$ measurements that, at least locally, suffices to
determine~$P$ up to congruence.

We restrict to convex polyhedra mostly for reasons of convenience: many of the
results below should be true in greater generality. (One problem with moving
beyond convex polyhedra is determining exactly what the term `polyhedron'
should mean: for a recent attempt to give a general definition of the term
`nonconvex polyhedron', see the beautiful paper \cite{grunbaum}.) To avoid any
ambiguity we begin with a precise definition of convex polyhedron.

\begin{definition}
A \emph{closed half-space} is a subset of $\mathbf{R}^{3}$ of the form
$\{\,(x, y, z)\in\mathbf{R}^{3}\mid ax + by + cz + d \ge0\,\}$ with $(a, b,
c)\ne(0, 0, 0)$.  A \emph{convex polyhedron} is a subset~$P$ of $\mathbf{R}%
^{3}$ which  is bounded, does not lie in any plane, and can be expressed as
an  intersection of finitely many closed half-spaces.
\end{definition}

The vertices, edges and faces of a convex polyhedron~$P$ can be defined in
terms of intersections of~$P$ with suitable closed half-spaces. For example, a
\emph{face} of~$P$ is a subset of~$P$ of the form $P\cap H$ for some closed
half-space~$H$, such that $P\cap H$ lies entirely within some plane but is not
contained in any line. Edges and vertices can be defined similarly.

The \hyperref[motivating_problem]{original problem} refers to two polyhedra
with the same `combinatorial structure', so we give a notion of \emph{abstract
polyhedron} which isolates the combinatorial information embodied in a convex polyhedron.

\begin{definition}
The underlying \emph{abstract polyhedron} of a convex polyhedron~$P$  is the
triple~$(\mathcal{V}_{P}, \mathcal{F}_{P}, \mathcal{I}_{P})$, where
$\mathcal{V}_{P}$ is the set of  vertices of~$P$, $\mathcal{F}_{P}$ is the set
of faces of~$P$, and $\mathcal{I}_{P}\subset\mathcal{V}_{P}\times
\mathcal{F}_{P}$ is the incidence relation between vertices and  faces; that
is, $(v, f)$ is in $\mathcal{I}_{P}$ if and only if the vertex~$v$  lies on
the face~$f$.
\end{definition}

Thus to say that two polyhedra~$P$ and~$Q$ have the same combinatorial
structure is to say that their underlying abstract polyhedra are
\emph{isomorphic}; that is, there are bijections~$\beta_{V}\colon
\mathcal{V}_{P}\rightarrow\mathcal{V}_{Q}$ and $\beta_{F}\colon\mathcal{F}%
_{P}\rightarrow\mathcal{F}_{Q}$ that respect the incidence relation: $(v,f)$
is in~$\mathcal{I}_{P}$ if and only if~$(\beta_{V}(v), \beta_{F}(f))$ is
in~$\mathcal{I}_{Q}$. Note that there is no need to record information about
the edges; we leave it to the reader to verify that the edge data and
incidence relations involving the edges can be recovered from the incidence
structure~$(\mathcal{V}_{P}, \mathcal{F}_{P}, \mathcal{I}_{P})$. The
cardinality of the set~$\mathcal{I}_{P}$ is twice the number of edges of~$P$,
since
\[
|\mathcal{I}_{P}| = \sum_{f\in\mathcal{F}_{P}}\text{(number of vertices
on~$f$)} = \sum_{f\in\mathcal{F}_{P}}\text{(number of edges on~$f$)}%
\]
and the latter sum counts each edge of~$P$ exactly twice.

For the remainder of this section, we fix a convex polyhedron~$P$ and write
$V$, $E$ and $F$ for the number of vertices, edges and faces of~$P$,
respectively. Let $\Pi= (\mathcal{V}, \mathcal{F}, \mathcal{I})$ be the
underlying abstract polyhedron. We are interested in determining which sets of
measurements are sufficient to determine~$P$ up to congruence. A natural place
to start is with a naive dimension count: how many degrees of freedom does one
have in specifying a polyhedron with the same combinatorial structure as~$P$?

\begin{definition}
A \emph{realization} of $\Pi= (\mathcal{V}, \mathcal{F}, \mathcal{I})$ is a
pair of functions  $(\alpha_{\V}, \alpha_{\F})$ where $\alpha_{\V}%
\colon\mathcal{V}\rightarrow\mathbf{R}^{3}$ gives a  point for each~$v$ in
$\mathcal{V}$, $\alpha_{\F}\colon\mathcal{F}\rightarrow\{\text{planes in
$\mathbf{R}^{3}$}\}$ gives a plane for each~$f$ in $\mathcal{F}$, and the
point~$\alpha_{\V}(v)$ lies on the plane~$\alpha_{\F}(f)$ whenever~$(v, f)$
is  in~$\mathcal{I}$.
\end{definition}

Given any convex polyhedron~$Q$ together with an isomorphism~$\beta
\colon(\mathcal{V}_{Q},\mathcal{F}_{Q},\mathcal{I}_{Q})\cong(\mathcal{V}%
,\mathcal{F},\mathcal{I})$ of incidence structures we obtain a realization
of~$\Pi$, by mapping each vertex of~$P$ to the position of the corresponding
(under~$\beta$) vertex of~$Q$ and mapping each face of~$P$ to the plane
containing the corresponding face of~$Q$. In particular,~$P$ itself gives a
realization of~$\Pi$, and when convenient we'll also use the letter~$P$ for
this realization. Conversely, while not every realization of~$\Pi$ comes from
a convex polyhedron in this way, any realization of~$\Pi$ that's
\emph{sufficiently close} to~$P$ in the natural topology for the space of
realizations gives---for example by taking the convex hull of the image
of~$\alpha_{V}$---a convex polyhedron whose underlying abstract polyhedron can
be identified with~$\Pi$. So the number of degrees of freedom is the dimension
of the space of realizations of~$\Pi$ in a neighborhood of~$P$.



Now we can count degrees of freedom. There are $3V$ degrees of freedom in
specifying~$\alpha_{\V}$ and $3F$ in specifying~$\alpha_{\F}$. So if the
$|\mathcal{I}| = 2E$ `vertex-on-face' conditions are independent in a suitable
sense then the space of all realizations of~$\Pi$ should have dimension~$3V +
3F - 2E$ or---using Euler's formula---dimension~$E+6$. We must also take the
congruence group into account: we have three degrees of freedom available for
translations, and a further three for rotations. Thus if we form the quotient
of the space of realizations by the action of the congruence group, we expect
this quotient to have dimension~$E$. This suggests that~$E$ measurements
should suffice to pin down~$P$ up to congruence.

In the remainder of this section we show how to make the above naive dimension
count rigorous, and how to identify specific sets of~$E$ measurements that
suffice to determine congruence. The main ideas are: first, to use a
combinatorial lemma (\hyperref[lem:order_faces_vertices]%
{Lemma~\ref*{lem:order_faces_vertices}}) to show that the linearizations of
the vertex-on-face conditions are linearly independent at~$P$, allowing us to
use the inverse function theorem to show that the space of realizations really
does have dimension~$E+6$ near~$P$ and to give an infinitesimal criterion for
a set of measurements to be sufficient (\hyperref[thm:sufficient]%
{Theorem~\ref*{thm:sufficient}}), and second, to use an infinitesimal version
of Cauchy's rigidity theorem to identify sufficient sets of measurements.

The various measurements that we're interested in can be thought of as
real-valued functions on the space of realizations of~$\Pi$ (defined at least
on a neighborhood of~$P$) that are invariant under congruence. We single out
one particular type of measurement: given two vertices~$v$ and $w$ of~$P$ that
lie on a common face, the \emph{face distance} associated to~$v$ and $w$ is
the function that maps a realization~$Q = (\alpha_{\V}, \alpha_{\F})$ of~$\Pi$
to the distance from~$\alpha_{\V}(v)$ to $\alpha_{\V}(w)$. In other words, it
corresponds to the measurement of the distance between the vertices of~$Q$
corresponding to~$v$ and $w$. The main result of this section is the following theorem.

\begin{theorem}
\label{thm:main}  Let~$P$ be a convex polyhedron with underlying abstract
polyhedron~$(\mathcal{V}, \mathcal{F}, \mathcal{I})$. Then there is a set~$S$
of face distances of~$P$ such that~(i) $S$ has cardinality~$E$, and  (ii)
locally near~$P$, the set~$S$ completely determines~$P$ up to  congruence in
the following sense: there is a positive real  number~$\varepsilon$ such that
for any convex polyhedron~$Q$ and  isomorphism~$\beta\colon(\mathcal{V},
\mathcal{F},\mathcal{I})\cong(\mathcal{V}_{Q}, \mathcal{F}_{Q}, \mathcal{I}%
_{Q})$  of underlying abstract polyhedra, if 

\begin{enumerate}

\item each vertex~$v$ of $P$ is within distance~$\varepsilon$ of the
corresponding vertex~$\beta_{\V}(v)$ of $Q$, and 

\item $m(Q) = m(P)$ for each measurement~$m$ in~$S$, 
\end{enumerate}

then $Q$ is congruent to $P$.
\end{theorem}

Rephrasing:

\begin{corollary}
Let~$P$ be a convex polyhedron with~$E$ edges. Then there is a set of~$E$
measurements that is sufficient to determine~$P$ up to congruence amongst  all
nearby convex polyhedra with the same combinatorial structure as~$P$.
\end{corollary}

We'll prove this theorem as a corollary of \hyperref[thm:sufficient]%
{Theorem~\ref*{thm:sufficient}} below, which gives conditions for a set of
measurements to be sufficient. We first fix some notation. Choose
numberings~$v_{1},\dots, v_{V}$ and $f_{1},\dots, f_{F}$ of the vertices and
faces of~$\Pi$, and write $(x_{i}(P), y_{i}(P), z_{i}(P))$ for the coordinates
of vertex~$v_{i}$ of~$P$. We translate~$P$ if necessary to ensure that no
plane that contains a face of~$P$ passes through the origin. This allows us to
give an equation for the plane containing~$f_{j}$ in the form $a_{j}(P)x +
b_{j}(P)y + c_{j}(P)z = 1$ for some nonzero triple of real numbers~$(a_{j}(P),
b_{j}(P), c_{j}(P))$; similarly, for any realization~$Q$ of~$\Pi$ that's close
enough to~$P$ the $i$th vertex of~$Q$ is a triple~$(x_{i}(Q), y_{i}(Q),
z_{i}(Q)$ and the $j$th plane of~$Q$ can be described by an equation~$a_{j}%
(Q)x + b_{j}(Q)y + c_{j}(Q)z = 1$. Hence the coordinate functions $(x_{1},
y_{1}, z_{1}, x_{2}, y_{2}, z_{2}, \dots, a_{1}, b_{1}, c_{1}, \dots)$ give an
embedding into~$\mathbf{R}^{3V+3F}$ of some neighborhood of~$P$ in the space
of realizations of~$\Pi$.

For every pair $(v_{i}, f_{j})$ in $\mathcal{I}$ a realization~$Q$ should
satisfy the `vertex-on-face' condition
\[
a_{j}(Q)x_{i}(Q) + b_{j}(Q)y_{i}(Q) + c_{j}(Q)z_{i}(Q) = 1.
\]
Let $\phi_{i,j}$ be the function from $\mathbf{R}^{3V+3F}$ to $\mathbf{R}$
defined by
\[
\phi_{i, j}(x_{1}, y_{1}, z_{1},\dots, a_{1}, b_{1}, c_{1}, \dots) = a_{j}
x_{i} + b_{j} y_{i} + c_{j} z_{i} -1,
\]
and let $\phi\colon\mathbf{R}^{3V + 3F}\rightarrow\mathbf{R}^{2E}$ be the
vector-valued function whose components are the $\phi_{i,j}$ as $(v_{i},
f_{j})$ runs over all elements of~$\mathcal{I}$ (in some fixed order). Then a
vector in $\mathbf{R}^{3V + 3F}$ gives a realization of~$\Pi$ if and only if
it maps to the zero vector under~$\phi$.

We next present a combinatorial lemma, \hyperref[lem:order_faces_vertices]%
{Lemma~\ref*{lem:order_faces_vertices}}, that appears as an essential
component of many proofs of Steinitz's theorem, characterizing edge graphs of
polyhedra. (See Lemma 2.3 of \cite{grunbaum}, for example.) We give what we
believe to be a new proof of this lemma. First, an observation that is an easy
consequence of Euler's theorem.

\begin{lemma}
Suppose that $\Gamma$ is a planar bipartite graph of order~$r$.  Then there is
an ordering~$n_{1}, n_{2}, \dots, n_{r}$ of the nodes of~$\Gamma$  such that
each node~$n_{i}$ is adjacent to at most three preceding nodes.
\end{lemma}

\begin{proof}
We give a proof by induction on~$r$.  If~$r \le 3$ then any ordering will
do.  If $r> 3$ then we can apply a standard consequence of Euler's formula
(see, for example, Theorem 16 of \cite{bollobas}), which states that the
number of edges in a bipartite planar graph of order~$r\ge 3$ is at most
$2r-4$.  If every node of~$\Gamma$ had degree at least~$4$ then the total
number of edges would be at least~$2r$, contradicting this result.  Hence
every nonempty planar bipartite graph has a node of degree at most~$3$;
call this node~$n_r$.  Now remove this node (and all incident edges),
leaving again a bipartite planar graph.  By the induction hypothesis, there
is an ordering~$n_1,\dots,n_{r-1}$ satisfying the conditions of the theorem,
and then $n_1,\dots, n_r$ gives the required ordering.
\end{proof}

\begin{lemma}
\label{lem:order_faces_vertices}  Let~$P$ be a convex polyhedron. Consider the
set $\mathcal{V}\cup\mathcal{F}$ consisting of  all vertices and all faces
of~$P$. It is possible to order the elements of  this set such that every
vertex or face in this set is incident with at most  three earlier elements of
$\mathcal{V}\cup\mathcal{F}$.
\end{lemma}

\begin{proof}
We construct a graph~$\Gamma$ of order~$V+F$ as follows.  $\Gamma$ has one
node for each vertex of~$G$ and one node for each face of~$G$.  Whenever a
vertex $v$ of $P$ lies on a face $f$ of $P$ we introduce an edge of $\Gamma$
connecting the nodes corresponding to $v$ and $f$.  Since $P$ is convex, the
graph $\Gamma$ is planar; indeed, by choosing a point on each face of $P$,
one can draw the graph $\Gamma$ directly on the surface of $P$ and then
project onto the plane.  (The graph~$\Gamma$ is known as the \emph{Levi
graph} of the incidence structure~$\Pi = (\V, \F, \I)$.)  Now apply the
preceding lemma to this graph.
\end{proof}

We now show that the functions~$\phi_{i, j}$ are independent in a neighborhood
of~$P$. Write $D\phi(P)$ for the derivative of~$\phi$ at~$P$; as usual, we
regard $D\phi(P)$ as a $2E$-by-$(3V+3F)$ matrix with real entries.

\begin{lemma}
\label{Dphirank} The derivative $D\phi(P)$ has rank~$2E$.
\end{lemma}

In more abstract terms, this lemma implies that the space of all realizations
of $\Pi$ is, in a neighborhood of~$P$, a smooth manifold of dimension~$3V + 3F
- 2E = E + 6$.

\begin{proof}
We prove that there are no nontrivial linear relations on the~$2E$
rows of~$D\phi(P)$.  To illustrate the argument, suppose that the
vertex $v_1$ lies on the first three faces and no others.  Placing
the rows corresponding to $\phi_{1,1}$, $\phi_{1,2}$ and
$\phi_{1,3}$ first, and writing simply~$x_1$ for $x_1(P)$ and
similarly for the other coordinates, the matrix $D\phi(P)$ has the
following structure.
\[D\phi(P) =
\left(
\begin{array}
[c]{ccc|c||ccccccccc|c}%
a_{1} & b_{1} & c_{1} & 0\dots0 & x_{1} & y_{1} & z_{1} & 0 & 0 & 0 & 0 & 0 &
0 & 0\dots0\\
a_{2} & b_{2} & c_{2} & 0\dots0 & 0 & 0 & 0 & x_{2} & y_{2} & z_{2} & 0 & 0 &
0 & 0\dots0\\
a_{3} & b_{3} & c_{3} & 0\dots0 & 0 & 0 & 0 & 0 & 0 & 0 & x_{3} & y_{3} &
z_{3} & 0\dots0\\\hline
0 & 0 & 0 &  &  &  &  &  &  &  &  &  &  & \\
\vdots & \vdots & \vdots & \ast &  &  &  &  & \ast &  &  &  &  & \ast\\
0 & 0 & 0 &  &  &  &  &  &  &  &  &  &  &
\end{array}
\right)
\]
Here the vertical double bar separates the derivatives for the vertex
coordinates from those for the face coordinates.  Since the faces $f_1$, $f_2$
and $f_3$ cannot contain a common line, the $3$-by-$3$ submatrix in the top
left corner is nonsingular.  Since~$v_1$ lies on no other faces, all other
entries in the first three columns are zero.  Thus any nontrivial linear
relation of the rows cannot involve the first three rows. So $D\phi(P)$ has
full rank (that is, rank equal to the number of rows) if and only if the
matrix obtained from $D\phi(P)$ by deleting the first three rows has full
rank---that is, rank~$2E-3$.
Extrapolating from the above, given any vertex that lies on exactly three
faces, the three rows corresponding to that vertex may be removed from the
matrix $D\phi(P)$, and the new matrix has full rank if and only if $D\phi(P)$
does. The dual statement is also true: exchanging the roles of vertex and face
and using the fact that no three vertices of~$P$ are collinear we see that for
any triangular face~$f$ we may remove the three rows corresponding to~$f$ from
$D\phi(P)$, and again the resulting matrix has full rank if and only if the
$D\phi(P)$ does. Applying this idea inductively, if \emph{every} vertex of $P$
lies on exactly three faces (as in for example the regular tetrahedron, cube
or dodecahedron), or dually if \emph{every} face of $P$ is triangular (as in
for example the tetrahedron, octahedron or icosahedron) then the lemma is
proved.
For the general case, we choose an ordering of the faces and vertices as in
\lemmaref{lem:order_faces_vertices}.  Then, starting from the top end of this
ordering, we remove faces and vertices from the list one-by-one, removing
corresponding rows of~$D\phi(P)$ at the same time.  At each removal, the new
matrix has full rank if and only if the old one does.  But after removing all
faces and vertices we're left with a~$0$-by-$2E$ matrix, which certainly has
rank~$0$.  So $D\phi(P)$ has rank~$2E$.
\end{proof}

We now prove a general criterion for a set of measurements to be sufficient.
Given the previous lemma, this criterion is essentially a direct consequence
of the inverse function theorem.

\begin{definition}
A \emph{measurement} for $P$ is a smooth  function~$m$ defined on an open
neighborhood of~$P$ in the space of  realizations of~$\Pi$, such that $m$ is
invariant under rotations  and translations.
\end{definition}

Given any such measurement~$m$, it follows from \hyperref[Dphirank]%
{Lemma~\ref*{Dphirank}} that we can extend~$m$ to a smooth function on a
neighborhood of $P$ in $\mathbf{R}^{3V+3F}$. Then the derivative~$Dm(P)$ is a
row vector of length~$3V + 3F$, well-defined up to a linear combination of the
rows in $D\phi(P)$.

\begin{theorem}
\label{thm:sufficient}  Let $S$ be a finite set of measurements for~$\Pi$
near~$P$. Let  $\psi\colon\mathbf{R}^{3V+3F}\rightarrow\mathbf{R}^{|S|}$ be
the vector-valued  function obtained by combining the measurements in~$S$,
and  write~$D\psi(P)$ for its derivative at~$P$, an $|S|$-by-$(3V+3F)$  matrix
whose rows are the derivatives~$Dm(P)$ for $m$ in~$S$. Then  the matrix
\[
D(\phi,\psi)(P) =
\begin{pmatrix}
D\phi(P)\\
D\psi(P)
\end{pmatrix}
\]
has rank at most~$3E$, and if it has rank  exactly~$3E$ then the measurements
in~$S$ are sufficient to  determine congruence: that is, for any
realization~$Q$ of $\Pi$,  sufficiently close to~$P$, if $m(Q) = m(P)$ for all
$m$ in $S$ then  $Q$ is congruent to~$P$.
\end{theorem}

\begin{proof} Let $Q(t)$ be any smooth one-dimensional family of
realizations of~$\Pi$ such that~$Q(0) = P$ and $Q(t)$ is congruent
to~$P$ for all~$t$.  Since each $Q(t)$ is a valid realization,
$\phi(Q(t))=0$ for all~$t$.  Differentiating and applying
the chain rule at~$t=0$ gives the matrix equation $D\phi(P) Q'(0) = 0$
where $Q'(0)$ is thought of as a column vector of length~$3V + 3F$.
The same argument applies to the map~$\psi$: since $Q(t)$ is
congruent to $P$ for all~$t$, $\psi(Q(t)) = \psi(P)$ is constant and
$D\psi(P) Q'(0) = 0$.
We apply this argument first to the three families where~$Q(t)$ is
$P$ translated~$t$ units along the~$x$-axis, $y$-axis, or $z$-axis
respectively, and second when~$Q(t)$ is $P$ rotated by~$t$ radians
around the $x$-axis, the $y$-axis and the $z$-axis.  This gives~$6$
column vectors that are annihilated by both $D\phi(P)$ and
$D\psi(P)$.  Writing~$G$ for the $(3V+3F)$-by-$6$ matrix obtained
from these column vectors, we have the matrix equation
$$\begin{pmatrix} D\phi(P)\\D\psi(P)
\end{pmatrix} G = 0.$$ It's straightforward to compute~$G$ directly;
we leave it to the reader to check that the transpose of~$G$ is
$$\left(\begin{array}[c]{ccccccc||ccccccc}
1& 0& 0& 1& 0& 0& \dots& -a_1^2& -a_1b_1& -a_1c_1& -a_2^2& -a_2b_2& -a_2c_2&\dots \\
0& 1& 0& 0& 1& 0& \dots& -a_1b_1& -b_1^2& -b_1c_1& -a_2b_2& -b_2^2& -b_2c_2& \dots \\
0& 0& 1& 0& 0& 1& \dots& -a_1c_1& -b_1c_1& -c_1^2& -a_2c_2& -b_2c_2& -c_2^2& \dots \\
0& -z_1& y_1& 0& -z_2& y_2& \dots& 0& -c_1& b_1& 0& -c_2& b_2& \dots \\
z_1& 0& -x_1& z_2& 0& -x_2&\dots& c_1& 0& -a_1& c_2& 0& -a_2&\dots \\
-y_1& x_1& 0& -y_2& x_2& 0&\dots& -b_1& a_1& 0& -b_2& a_2& 0&\dots
\end{array}\right)$$
For the final part of this argument, we introduce a notion of
normalization on the space of realizations of~$\Pi$.  We'll say that
a realization~$Q$ of~$\Pi$ is \emph{normalized} if $v_1(Q) =
v_1(P)$, the vector from~$v_1(Q)$ to $v_2(Q)$ is in the direction of
the positive~$x$-axis, and~$v_1(Q)$, $v_2(Q)$ and $v_3(Q)$ all lie
in a plane parallel to the $xy$-plane, with $v_3(Q)$ lying in the
positive $y$-direction from~$v_1(P)$ and $v_2(P)$.  In terms of
coordinates we require that $x_1(P) = x_1(Q) < x_2(Q)$, $y_1(P) =
y_1(Q) = y_2(Q) < y_3(Q)$ and $z_1(P) = z_1(Q) = z_2(Q) = z_3(Q)$.
Clearly every realization~$Q$ of~$\Pi$ is congruent to a unique
normalized realization, which we'll refer to as the
\emph{normalization} of~$Q$.  Note that the normalization operation
is a continuous map on a neighborhood of~$P$.
Without loss of generality, rotating~$P$ around the
origin if necessary, we may assume that~$P$ itself is normalized.
The condition that a realization~$Q$ be normalized gives six more
conditions on the coordinates of~$Q$, corresponding to six extra
functions $\chi_1,\dots, \chi_6$, which we use to augment the
function $(\phi,\psi)\colon \R^{3V + 3F}\rightarrow \R^{2E + |S|}$
to a function $(\phi,\psi,\chi)\colon \R^{3V + 3F}\rightarrow \R^{2E
+ |S| + 6}$.  These six functions are simply~$\chi_1(Q) =
x_1(Q)-x_1(P)$, $\chi_2(Q) = y_1(Q) - y_1(P)$, $\chi_3(Q) = z_1(Q) -
z_1(P)$, $\chi_4(Q) = y_2(Q) - y_1(P)$, $\chi_5(Q) = z_2(Q)-z_1(P)$
and $\chi_6(Q) = z_3(Q) - z_1(P)$.
\begin{claim} The $6$-by-$6$ matrix $D\chi(P) G$ is invertible.
\end{claim}
\begin{proof} The matrix for~$G$ was given earlier; the
  product~$D\chi(P)G$ is easily verified to be
$$\begin{pmatrix}
1 & 0 & 0 & 0    & z_1  & -y_1 \\
0 & 1 & 0 & -z_1 & 0    & x_1  \\
0 & 0 & 1 & y_1  & -x_1 & 0    \\
0 & 1 & 0 & -z_2 & 0    & x_2  \\
0 & 0 & 1 & y_2  & -x_2 & 0    \\
0 & 0 & 1 & y_3  & -x_3 & 0    \\
\end{pmatrix}$$
which has nonzero determinant $(y_3-y_1)(x_2-x_1)^2$.
\end{proof}
As a corollary, the columns of~$G$ are linearly independent, which
proves that the matrix in the statement of the theorem has rank at
most~$3E$.  Similarly, the rows of $D\chi(P)$ must be linearly
independent, and moreover no nontrivial linear combination of those
rows is a linear combination of the rows of $D\psi(P)$.  Hence if
$D\psi(P)$ has rank exactly~$3E$ then the augmented matrix
$$\begin{pmatrix}
D\phi(P) \\
D\psi(P) \\
D\chi(P)
\end{pmatrix}$$ has rank~$3E+6$.  Hence the map $(\phi, \psi, \chi)$
has injective derivative at~$P$, and so by the inverse function
theorem the map $(\phi, \psi, \chi)$ itself is injective on a
neighborhood of~$P$ in~$\R^{3V+3F}$.
Now suppose that $Q$ is a polyhedron as in the statement of the
theorem.  Let $R$ be the normalization of~$Q$.  Then $\phi(R) =
\phi(Q) = \phi(P) = 0$, $\psi(R) = \psi(Q) = \psi(P)$, and $\chi(R)
= \chi(P)$.  So if~$Q$ is sufficiently close to~$P$, then by
continuity of the normalization map~$R$ is close to~$P$ and hence~$R
= P$ by the inverse function theorem.  So~$Q$ is congruent to~$R = P$ as
required.
\end{proof}

\begin{definition}
Call a set~$S$ of measurements \emph{sufficient}  for~$P$ if the conditions of
the above theorem apply: that is, the  matrix $D(\phi,\psi)(P)$ has rank~$3E$.
\end{definition}

\begin{corollary}
\label{cor:sizeE} Given a sufficient set~$S$ of  measurements, there's  a
subset of $S$ of size~$E$ that's also sufficient.
\end{corollary}

\begin{proof} Since the matrix $D(\phi,\psi)(P)$ has rank~$3E$ by assumption,
and the~$2E$ rows coming from $D\phi(P)$ are all linearly
independent by \lemmaref{Dphirank}, we can find~$E$ rows
corresponding to measurements in~$S$ such that $D\phi(P)$ together
with those~$E$ rows has rank~$3E$.
\end{proof}

The final ingredient that we need for the proof of
\hyperref[thm:main]%
{Theorem~\ref*{thm:main}} is the infinitesimal version of Cauchy's
rigidity theorem, originally due to Dehn, and later given a simpler
proof by Alexandrov.  We phrase it in terms of the notation and
definitions above.

\begin{theorem}
\label{thm:dehn}
Let $P$ be a convex polyhedron, and suppose that $Q(t)$ is a continuous
family of polyhedra specializing to~$P$ at  $t=0$. Suppose that the
real-valued function $m\circ Q$ is  stationary (that is, its derivative
vanishes) at $t=0$ for each face distance~$m$. Then $Q^{\prime}(0)$ is in
the span of the columns of~$G$  above.
\end{theorem}

\begin{proof}
  See Chapter 10, section 1 of \cite{alexandrov}.  See also section 5
  of \cite{Gluck} for a short self-contained version of Alexandrov's
  proof.
\end{proof}

\begin{corollary}
\label{cor:cauchy} The set of all face distances is sufficient.
\end{corollary}

\begin{proof}
  Let $S$ be the collection of all face distances. Then the matrix
  $D(\phi, \psi)(P)$ has rank~$3E$.  Now \theoremref{thm:dehn} implies
  that for any column vector~$v$ such that $D(\phi,\psi)(P) v = 0$,
  $v$ is in the span of the columns of~$G$.  Hence the kernel of the
  map $D(\phi,\psi)(P)$ has dimension exactly~$6$ and so by the
  rank-nullity theorem together with Euler's formula the rank of
  $D(\phi,\psi)(P)$ is $3V+3F-6 = 3E$.
\end{proof}

Now \hyperref[thm:main]%
{Theorem~\ref*{thm:main}} follows from \hyperref[cor:cauchy]%
{Corollary~\ref*{cor:cauchy}} together with \hyperref[cor:sizeE]%
{Corollary~\ref*{cor:sizeE}}.

Finding an explicit sufficient set of~$E$ face distances is now a
simple matter of turning the proof of \ref{cor:sizeE} into a
constructive algorithm.  First compute the matrices $D\phi(P)$ and
$D\psi(P)$ (the latter corresponding to the set~$S$ of \emph{all} face
distances).  Initialize a variable~$T$ to the empty set.  Now iterate
through the rows of $D\psi(P)$: for each row, if that row is a linear
combination of the rows of $D\phi(P)$ and the rows of $D\psi(P)$
corresponding to measurements already in~$T$, discard it.  Otherwise,
add the corresponding measurement to~$T$.  Eventually,~$T$ will be a
sufficient set of~$E$ face distances.

A computer program has been written in the language Python to
implement this algorithm.  This program is attached as Appendix B of
the online version of this paper, available at www.arXiv.org. 
We hope that the comments within the program are sufficient explanation
for those who wish to try it, and we thank Eric Korman for a number of
these comments.

The algorithm works more generally.  Given any sufficient set of measurements,
 which may include angles, it will extract a subset of 
$E$ measurements which is sufficient.
It will also find a set of angle measurements which determines a polyhedron up 
to similarity. As an example we consider a similarity calculation for a 
dodecahedron. Our
allowed measurements will be the set of angles formed by pairs of lines from a
vertex to two other points on a face containing that vertex. In Figure 2 we
show a set of $29$ such angles which the program determined to characterize a
dodecahedron up to similarity.

\begin{center}
{\parbox[b]{1.708in}{\begin{center}
\includegraphics[
natheight=2.698200in,
natwidth=2.500200in,
height=1.8421in,
width=1.708in
]{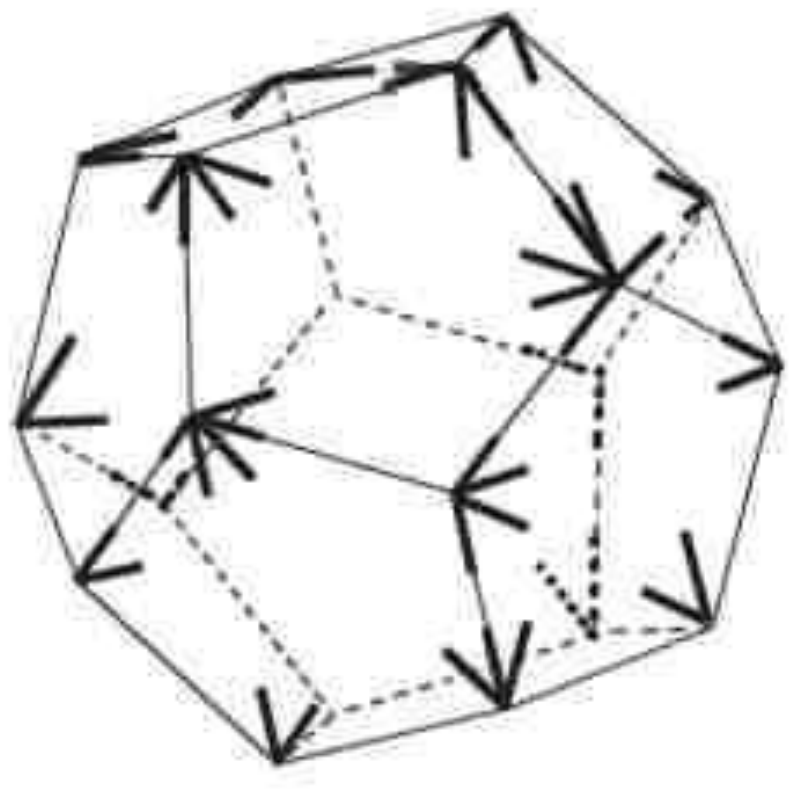}\\
Figure 2.  A set of 29 angles which locally determine a dodecahedron up to
similarity.
\end{center}}}
\end{center}

There is no restriction of the method to Platonic solids.  Data for a
number of other examples can be found in the program listing. Among
these examples are several which are non-convex. The program appears
to give a result for many of these, but we have not extended the
theory beyond the convex case.

\subsection{Related work}

The results and ideas of section~2 are, for the most part, not new. The idea
of an abstract polyhedron represented by an incidence structure, and its
realizations in $\mathbf{R}^{3}$, appears in section~2 of~\cite{whiteley}. In
Corollary~15 of that paper, Whiteley proves that the space of realizations of
a `spherical' incidence structure (equivalent to an incidence structure
arising from a convex polyhedron) has dimension~$E$. The essential
combinatorial content of the proof of \hyperref[Dphirank]%
{Lemma~\ref*{Dphirank}} is often referred to as `Steinitz's lemma', and 
a variety of proofs appear in the literature (\cite{Lyusternik}, \cite{grunbaumbook}); we
believe that the proof above is new.

\section{How many measurements are enough?}

\hyperref[thm:main]{Theorem~\ref*{thm:main}} provides an upper bound for the
number of distance/angle measurements needed to describe a polyhedron with
given combinatorial structure. But it turns out that many interesting
polyhedra can be described with fewer measurements. In particular, a cube can
be determined by $9$ distance/angle measurements instead of $12.$ Also, $10$
distance measurements (no angles used) suffice.

This phenomenon appears already in dimension two. Much of this section deals
with polygons and sets of points in the plane. In this simpler setting one can
often find precisely the smallest number of measurements needed. Extending
these results to polyhedra, with or without fixing their combinatorial
structure, certainly deserves further study.

So, we turn our attention to polygons. Unlike polyhedra, their combinatorial
structure is relatively simple, especially if we restrict to the case of
convex polygons. One can argue about the kinds of measurements that should be
allowed, but the following three seem the most natural to us.

\begin{definition}
Suppose that $A_{1}A_{2}\cdot\cdot\cdot A_{n}$ is a convex polygon. (More
generally, suppose $\{A_{1},A_{2},...,A_{n}\}$ is a sequence of distinct points on the
plane).  Then the following are called simple measurements:

\begin{enumerate}
\item distances $\left\vert A_{i}A_{j}\right\vert ,$ $i\neq j$

\item angles $\angle A_{i}A_{j}A_{k}$ for $i,j,k$ distinct

\item `diagonal angles' between  $A_{i}A_{j}$ and $A_{k}A_{l}$.
\end{enumerate}
\end{definition}

Other quantities might also be considered, like the distance from $A_{i}$ to
$A_{j}A_{k}$, but in practice these would require several measurements.

It is natural to ask how many simple measurements are needed to determine a
polygon up to isometry. In the case of triangle the answer is $3.$ There are a
couple of different ways to do it, which are very important for the classical
Euclidean Geometry (SSS, SAS, SSA, ASA, AAS). It is appropriate to note that
two sides of a triangle and an angle adjacent to one of the sides do not
always determine the triangle uniquely. But they do determine it locally (as
in part ii of \hyperref[thm:main]{Theorem~\ref*{thm:main}}).

It is easy to see that for any $n-$gon $P=A_{1}\cdot\cdot\cdot A_{n}$ the
following $(2n-3)$ measurements suffice: the $n-1$ distances $\left\vert
{A_{1}A_{2}}\right\vert $, $\left\vert {A_{2}A_{3}}\right\vert $, \dots,
$\left\vert {A_{n-1}A_{n}}\right\vert $, together with the~$n-2$ angles
$\angle A_{1}A_{2}A_{3}$, \dots, $\angle A_{n-2}A_{n-1}A_{n}$. Observe also
that instead of using distances and angles, one can use only distances, by
substituting $\left\vert A_{i-1}A_{i+1}\right\vert $ for $\angle A_{i-1}%
A_{i}A_{i+1}$.

The following theorem implies that for most polygons one can not get away with
fewer measurements. To make this rigorous, we will assume that $A_{1}=\left(
0,0\right)  ,$ and that $A_{2}=\left(  x_{2},0\right)  $ for some $x_{2}>0$.
It then becomes obvious that to each polygon we can associate a point in
$\mathbf{R}^{2n-3}$, corresponding to the undetermined coordinates
$x_{2},x_{3},y_{3},\dots,x_{n},y_{n}$ with $\left(  x_{i},y_{i}\right)
\neq\left(  x_{j} ,y_{j}\right)  $ for $i\neq j$.

\begin{theorem}
\label{th3-1} Denote by $S$ the set of all points in $\mathbf{R}^{2n-3}$
obtained by  the procedure above from those polygons that can be determined up
to  isometry by fewer than $2n-3$ measurements. Then $S$ has measure zero.
\end{theorem}

\begin{proof}
Observe that any set
of $2n-3$ specific measurements is a smooth map from $\R^{2n-3}$ to
itself. There is only a finite number of such maps, with
measurements chosen from the the types 1,2 and 3. \ At a non-critical
point, the map has an inverse.  The result is then a consequence of
Sard's theorem (Chapter 2, Theorem 8 of \cite{spivak}).
\end{proof}

The above theorem is clearly not new, and we claim no originality for its
statement or proof. For $n=3$ it implies that for a generic triangle one needs
at least three simple measurements. This is true for any triangle, because two
measurements are clearly not enough. For $n=4$ the theorem implies that a
generic convex polygon requires $5$ measurements. One can be tempted to
believe that no quadrilateral can be described by fewer than $5$ measurements.
In fact, the first reaction of most mathematicians, seems to be that if a
general case requires $5$ measurements, all special cases do so as well. The
following example seems to confirm this intuition.

\begin{example}
Suppose $P=A_{1}A_{2}A_{3}A_{4}$ is a quadrilateral. Suppose $\angle
A_{2}A_{1}A_{4} =\alpha_{1},$ $\angle A_{3}A_{4}A_{1}=\alpha_{2},$
$|A_{1}A_{2}|=d_{1},$  $|A_{3}A_{4}|=d_{2},$ and $\angle A_{1}A_{2}%
A_{3}=\alpha_{3}.$ For most polygons $P$ these  five measurements are
sufficient to determine $P$ up to congruence. However  for some $P$ they are
not sufficient. For example, if $P$ is a square, there  are infinitely many
polygons with the same measurements: all rectangles with  the same
$|A_{1}A_{2}|.$
\end{example}

The above example suggests that the special polygons require at least as many
measurements as the generic ones. So $(2n-3)$ should be the smallest number of
simple measurements required for any $n-$gon. However, despite the above
example, this is very far from the truth: many interesting $n-$gons can be
determined by fewer than $(2n-3)$ simple measurements.

The simplest example in this direction is not usually considered a polygon.
Suppose $A_{1}A_{2}A_{3}A_{4}$ is a set of $4$ points on the plane that lie on
the same line, in the natural order. Then the $4$ distances $|A_{1}A_{4}|,$
$|A_{1}A_{2}|,$ $|A_{2}A_{3}$ $|A_{3}A_{4}|$ determine the configuration up to
the isometry. This means that any subset of four points on the plane with the
same corresponding measurements is congruent to $A_{1}A_{2}A_{3}A_{4}.$ The
proof of this is, of course, a direct application of the triangle inequality:
\[
|A_{1}A_{4}|\leq|A_{1}A_{3}|+|A_{3}A_{4}| \leq|A_{1}A_{2}| +|A_{1}%
A_{2}|+|A_{2}A_{3}|+|A_{3}A_{4}|
\]
One can think of $A_{1}A_{2},$ $A_{2}A_{3},$ $A_{3}A_{4},$ and $A_{1}A_{4}$ as
of rods of fixed length. When $|A_{1}A_{4}|=|A_{1}A_{2}| +|A_{1}A_{2}%
|+|A_{2}A_{3}|+|A_{3}A_{4}|,$ the configuration allows no room for wiggling,
all rods must be lined up. This degenerate example easily generalizes to $n$
points and $n$ distance measurements. More surprisingly, it generalizes to
some convex polygons.

\begin{definition}
A polygon is called \textit{exceptional} if it can be  described locally by
fewer than $2n-3$ measurements.
\end{definition}

Obviously, no triangles are exceptional, in the above sense. But already for
the quadrilaterals, there exist some exceptional ones, that could be defined
by four rather than five measurements.
%
The biggest surprise for us was the following observation.

\begin{proposition}
All squares are exceptional!

Specifically, for a square $ABCD$ the following four measurements distinguish
it among all quadrilaterals:
\[
|AB|, |AC|, |AD|, \angle BCD.
\]

\end{proposition}

\begin{proof}
Suppose $A'B'C'D'$ is another quadrilateral with $|A'B'|=|AB|,$
$|A'C'|=|AC|,$ $|A'D'|=|AD|,$ and $\angle B'C'D' = \angle BCD$. If $|AB|=d,$
this means that $|A'B'|=|A'D'|=d,$ $|A'C'|=d\sqrt{2},$ and $\angle B'C'D'
=\frac{\pi}{2}.$ We claim that the quadrilateral $A'B'C'D'$ is congruent to
the square $ABCD.$ Indeed, consider $A'$ fixed. Then $B'$ and $D'$ lie on
the circle of radius $d$ centered at $A',$ while $C'$ lies on the circle of
radius $d\sqrt{2}$ centered at $A'.$ Then $\frac{\pi}{2}$ is the maximal
possible value of the angle $B'C'D'$ and it is achieved when $\angle A'B'C'
= \angle A'D'C' =\frac{\pi}{2},$ making $A'B'C'D'$ a square, with the side
$|A'B'|=d.$
\end{proof}

The idea of showing that a particular polygon occurs when some angle or length
is maximized within given constraints, and that there is only one such
maximum, is at the heart of all of the examples in this section. It is related
to the notion of second order rigidity of tensegrity networks, see
\cite{ConWhi1} and \cite{RothWhi}.

One immediate generalization of this construction is the following
$3$-parameter family of exceptional quadrilaterals.

\begin{proposition}
\label{prop:exceptional_quadrilateral}

\label{ex3}Suppose that $ABCD$ is a convex quadrilateral with $\angle ABC =
\angle CDA =\frac{\pi}{2}.$ Then $ABCD$ is exceptional. It is determined up
to congruence by the following four measurements: $\left\vert AB\right\vert ,$
$\left\vert AD\right\vert ,$ $\left\vert AC\right\vert ,$ and $\angle BCD  .$
\end{proposition}

\begin{proof}
Consider an arbitrary quadrilateral $ABCD$ with the same four
measurements as the quadrilateral that we are aiming for (note that
we are not assuming that the $\angle ABC = \angle CDA
=\frac{\pi}{2}$).  We can consider the points $A$ and $C$ fixed on
the plane. Then the points $B$ and $D$ lie on two fixed circles
around $A,$ with radius $\left\vert AB\right\vert $ and radius
$\left\vert AD\right\vert $.  Among all such pairs of points on this
circle, on opposite sides of $AC$, the maximum possible value of
$\angle BCD$ is obtained for only one choice of $B$ and $D$.  This
is the choice which makes $CB$ and $CD$ tangent to the circle at $B$
and $D$, and so $\angle ABC$ and $\angle ADC$ are right angles, and
the quadrilateral $ABCD$ is congruent to the one we are aiming for.
This family of quadrilaterals includes all rectangles.  In a sense
it is the biggest possible: one cannot hope for a four-parameter
family requiring just four measurements.  Another such family is
given below.  Note that it includes all rhombi that are not squares.
\begin{proposition}
Suppose $ABCD$ is a convex quadrilateral, and for given $B$ and
$D,$ and given acute angles $\theta_{1},\theta_{2},$ choose $A$
and $C$ so that $\angle DAB=\theta_{1},$ $\angle DCB=\theta_{2}$,
and $\left\vert AC\right\vert $ is as large as possible.  Then
this determines a unique quadrilateral, up to congruence, and this
quadrilateral has the property that $\left\vert AB\right\vert
=\left\vert AD\right\vert $ and $\left\vert CB\right\vert
=\left\vert CD\right\vert .$
\end{proposition}
\end{proof}

We leave the proof to the reader. It implies that for the set of
quadrilaterals $ABCD$ such that $AC$ is a perpendicular bisector of $BD,$ and
the angles $\angle DAB$ and $\angle DCB$ are acute, the measurements
$\left\vert BD\right\vert ,\left\vert AC\right\vert ,$ $\angle DAB$ and
$\angle DCB$ are sufficient to determine $ABCD$.

As David Allwright pointed out to the authors, one can further extend this
example to the situation when $\angle DAB + \angle DCB <\frac{\pi}{2}.$

The existence of so many exceptional quadrilaterals suggested that for bigger
$n$ it may be possible to go well below the $(2n-3)$ measurements. This is
indeed correct, for every $n$ there are polygons that can be defined by just
$n$ measurements.

\begin{proposition}
Suppose that $A_{1}A_{2}\cdot\cdot\cdot A_{n}$ is a convex polygon, with
\[
\angle A_{1}A_{2}A_{3}=\angle A_{1}A_{3}A_{4}=\dots=\angle A_{1}A_{k}A_{k+1}%
=\dots=\angle A_{1}A_{n-1}A_{n}=\frac{\pi}{2}%
\]

(Note that many such polygons exist for every $n$). Then $A_{1}A_{2}\cdot
\cdot\cdot A_{n}$ is exceptional, moreover the distances $|A_{1}A_{2}|,$
$|A_{1}A_{n}|$ and the angles $\angle A_{k}A_{k+1}A_{1},$ $2\leq k\leq{n-1}$
determine the polygon.
\end{proposition}

\textbf{Proof.} Suppose $\angle A_{k}A_{k+1}A_{1}=\alpha_{k},$ $2\leq
k\leq(n-1).$ Then because of the Law of Sines for the triangles $\triangle
A_{1}A_{k}A_{k+1},$ we obtain the following sequence of inequalities:
\[
|A_{1}A_{n}|\leq\frac{|A_{1}A_{n-1}|}{\sin\alpha_{n-1}}\leq\frac{|A_{1}%
A_{n-2}|}{\sin\alpha_{n-1}\cdot\sin\alpha_{n-2}}\leq\dots\leq\frac{|A_{1}%
A_{2}|}{\sin\alpha_{n-1}\cdot\dots\cdot\sin\alpha_{2}}%
\]

Equality is achieved if and only if all angles $A_{1}A_{k}A_{k+1}$ are right
angles, which implies the result.

One can ask whether an even smaller number of measurements might work for some
very special polygons. The following theorem shows that in a very strong sense
the answer is negative.

\begin{theorem}
\label{th3-2}For any sequence of distinct points  $A_{1},A_{2},\dots,A_{n}$
on the plane, one needs at least $n$ distance  / angle / diagonal angle
measurements to determine it up to  plane isometry.
\end{theorem}

\begin{proof}
At least one distance measurement is needed, and we can assume that
it is $\left\vert A_{1}A_{2}\right\vert $.  We assume
that $A_{1}A_{2}$ is fixed, so the positions of the other $n-2$
points determine the set up to isometry.  We identify the ordered
set of coordinates of these points with a point in
$\R^{2(n-2)}=\R^{2n-4}$.  The set of all sequences of distinct points then
corresponds to an open set $U\subset \R^{2n-4}$.  Each measurement of
type 1, 2, or 3 determines a smooth submanifold $V\subset U$.
The following observation is the main idea of the proof.
\begin{lemma}
\label{lem3-1}Suppose that $x\in V$.  Then there exists an affine subspace
$W$ of $\R^{2n-4},$ of dimension at least $2n-6,$ which contains $x$ and is
such that for some open ball $B$ containing $x$ we have $W\cap B\subset V$.
\end{lemma}
\end{proof}

\begin{proof}
The proof of this lemma involves several different cases, depending on the
kind of measurements used and whether $A_1$ and/or $A_2$  are
involved.  We give some examples, the other cases being similar.
First suppose that the polygon is the unit square, with vertices at
$A_{1}=\left(  0,0\right)  ,$ $A_{2}=\left(  1,0\right)  ,$ $A_{3}=\left(
1,1\right)  ,$ and $A_{4}=\left(  0,1\right)  $   Suppose that the
measurement is the angle $\angle A_2A_1A_3$.  With $A_{1}$
and $A_{2}$ fixed, we are free to move $A_{3}$ along the line $y=x$,  and
$A_{4}$ arbitrarily.  In this case, then, $W$  could be three dimensional,
one more than promised by the Lemma.
Second, again with the unit square, suppose that the measurement is the
distance from $A_{3}$ to $A_{4}$.   In this case, we can move $A_{3}$ and
$A_{4}$ the same distance along parallel lines.  Thus,
\[
W=\left\{  \left(  1+c,1+d,c,1+d\right)  \right\}
\]
which is the desired two-dimensional affine space.
Finally, suppose that $n\geq5$  and that the measurement is the angle between
$A_{1}A_{k}$ and $A_{l}A_{m}$ where $1,2,k,l,m$  are
distinct.  In this case we can move $A_{l}$ arbitrarily,  giving two free
parameters, and we can move $A_{m}$ so that $A_{l}A_{m}$ remains parallel to
the original line containing these points. Then the length $\left\vert
A_{l}A_{m}\right\vert $ can be changed.  This gives a third parameter.  Then
the length $A_{1}A_{k}$ can be changed, giving a fourth, and the remaining
$n-5$ points can be moved, giving $2n-10$  more dimensions.  In this case,
$W$  is $2n-6$ dimensional.  We leave other cases to the reader.
\end{proof}

Continuing the proof of Theorem \ref{th3-2}, we first show that $k\leq n-3$
measurements is insufficient to determine the points $A_{1},\dots,A_{n}$.
Recall that one measurement, $\left\vert A_{1}A_{2}\right\vert $, was already
used. Suppose that a configuration $x\in\mathbf{R}^{2n-4}$ belongs to all the
varieties $V_{1},\dots,V_{k}.$ We will prove that $x$ is not an isolated point
in $V=\cap_{i=1}^{k}V_{i}$.

\begin{proof}
Denote the affine subspaces obtained from Lemma \ref{lem3-1} corresponding to
$V_{i}$ and $x$ by $W_{i}$.  Let $W=W_{1}\cap\cdot\cdot\cdot\cap W_{k}$.
Since each $W_{i}$ has dimension at least $2n-6$, its codimension in
$\R^{2n-4}$ is $ $less than or equal to $2.\footnote{If $V$ is an affine
subspace of $\R^{m}$,  then it is of the form $x+X$,  where $X$ is a subspace
of $\R^{m}.$ If the orthogonal complement of $X$ in $\R^{m}$ is $Y$,  then the
codimension of $V$ is the dimension of $Y$.  }$   Hence%
\[
\text{codim}W\leq\text{codim }W_{1}+\cdot\cdot\cdot+\text{codim }W_{k}%
\leq2k\leq2\left(  n-3\right)  <2n-4,
\]
so $W$ must have dimension at least $2.$ Since $W$ is contained in all of the
varieties $V_1,..,V_k,$ there is a two-dimensional affine subspace of points
containing $x$  and satisfying all of the measurements.  A neighborhood of
$x$ in this subspace is contained in $U$,  showing that $x$ is not isolated.
The case $k=n-2$ is trickier, because the dimensional count does not work in
such a simple way.  In this case, it is possible that $W=W_{1}\cap\cdot\cdot\cdot\cap
W_{k}=\left\{  x\right\}  $.   However, $W^{\prime}=W_{2}\cap\cdot\cdot
\cdot\cap W_{k}$ must have dimension at least $2$.
We now consider the space $W_{1}$ in more detail.  Denote the measurement
defining $V_{1}$ by $f_{1}:\R^{2n-4}\rightarrow \R$.   The linearization of
$f_{1}$ at $x$ is its derivative, $Df_{1}\left(  x\right)  .$  For
measurements of types 1, 2, or 3 above it is not hard to see that
$Df_{1}\left(  x\right)  \neq\mathbf{0}$.   If $X$ is the null space of the
$1\times(2n-4)$ matrix $Df_{1}\left(  x\right)  $,  then the tangent space of
$V_{1}$ at $x$ can be defined as the affine subspace $T_{x}\left(
V_{1}\right)  =x+X$.   This has dimension $2n-5$.
It is possible that $W_{1}=T_{x}\left(  V_{1}\right)  .$ This is the case in
the first example in the proof of the Lemma above.  Then, $\dim W_{1}=2n-5$.   Even if
$\dim\left(  W^{\prime}\right)  =2,$ we must have dim$\left(  W\right)  \geq
1$,  and since $W\cap U\subset V,$ $x$ is not isolated in $V$.
However in the second example in the Lemma above, (a distance measurement for the square),
$\dim W_{1}=2<\dim T_{x}\left(  V_{1}\right)  =3.$   Suppose we add as the
last two measurement the distance from $A_{4}$ to $A_{2}$.  Then one can
easily work out the spaces exactly, and show that $V_{1}\cap W^{\prime}$, when
projected onto the first three coordinates $(x_{3},y_{3},x_{4})$,  is the
intersection of a vertical cylinder with the $x_{3},x_{4}$ plane.  Again, $x$
is not isolated in $V$.
We have left to consider the general case, where $\dim W_{1}=2n-6$.  (We
are assuming that in the lemma we always choose the maximal $W$.)  Since
$\dim T_{x}\left(  V_{1}\right)  =2n-5$ and $\dim W^{\prime}=2$,
$W^{\prime}\cap T_{x}\left(  V_{1}\right)  $  must be one dimensional.
We wish to show that $x$ is not isolated in $V_{1}\cap W^{\prime}.$  We
consider the map $g_{1}=f_{1}|_{W^{\prime}}$.  If $Dg_{1}\neq0$,  then the
implicit function theorem implies that $0$ is not an isolated zero of $g_{1}$
in $W^{\prime}$,  proving the theorem.
To show that $Dg_{1}\neq0$ we can use the chain rule.  Let $i:W^{\prime
}\rightarrow \R^{2n-4}$  be the identity on $W^{\prime}$.  Then $g_{1}%
=f_{1}\circ i$  and so by the chain rule, $Dg_{1}=Df_{1}\circ i.$  If
$Dg_{1}=0,$ then $\left(  Df_{1}\right)  |_{W^{\prime}}=0,$ which implies
that $W^{\prime}\subset T_{x}\left(  V_{1}\right)  .$ Since $\dim W^{\prime
}=2,$ this contradicts the earlier assertion that $W^{\prime}\cap
T_{x}\left(  V_{1}\right)  $ is one dimensional, completing the proof of the
theorem.
\end{proof}

Similar ideas can be used to construct other interesting examples of
exceptional polygons and polyhedra. The following are worth mentioning.

\begin{enumerate}
\item There exist tetrahedra determined by just $5$ measurements,
  instead of the generic $6$. In particular, if measurement of
  dihedral angles is permitted, the regular tetrahedron can be
  determined using $5$ measurements.  It seems unlikely that $4$ would
  ever work.

\item There exist $5$-vertex convex polyhedra that are characterized by just
$5$ measurements. Moreover, four of the vertices are on the same plane, but,
unlike in the beginning of the paper, we do not have to assume this a  priori!
To construct such an example, we start with an exceptional  quadrilateral
$ABCD$ as in Proposition  \hyperref[prop:exceptional_quadrilateral]%
{Proposition~\ref*{prop:exceptional_quadrilateral}}, with $|AD|<|AB|.$ We
then  add a vertex $E$ outside of the plane $ABCD$, so that $\angle ADE=\angle
AEB=\frac{\pi}{2}.$ There are obviously many such polyhedra. Now notice that
the distances $|AD|,$ $|AC|,$ and angles $\angle AED,$ $\angle ABE,$ $\angle
BCD$ completely determine the configuration.

\item As pointed out earlier, using four measurements for a square, one can
determine the cube with just $9$ distance / angle measurements. Interestingly,
one can also use $10$ distance measurements for a cube. If $A_{1}B_{1}%
C_{1}D_{1}$ is its base and $A_{2}B_{2}C_{2}D_{2}$ is a parallel face, then
the six distances between $A_{1},$ $B_{1},$ $D_{1}$ and $A_{2}$ completely fix
the corner. Then $|A_{1}C_{2}|,$ $|B_{2}C_{2}|,$ $|D_{2}C_{2}|,$ and
$|C_{1}C_{2}|$ determine the cube, because
\[
|A_{1}C_{2}|^{2}\leq|B_{2}C_{2}|^{2}+|D_{2}C_{2}|^{2}+|C_{1}C_{2}|^{2},
\]
with the equality only when the segments $B_{2}C_{2},$ $D_{2}C_{2},$ and
$C_{1}C_{2}$ are perpendicular to the corresponding faces of the tetrahedron
$A_{1}B_{1}D_{1}A_{2}$ $\bigskip$ $.$
\end{enumerate}

The last example shows that the problem is not totally trivial even when we
only use the distance measurements. It is natural to ask for the smallest
number of distance measurements needed to fix a non-degenerate $n-$gon in the
plane. We have the following theorem.

\bigskip

\begin{theorem}
Suppose $A_{1},A_{2},\dots,A_{n}$ are points in the plane, with no three of
them on the same line. Then one needs at least $\min(2n-3,\frac{3n}{2})$
distance measurements to determine it up to a plane isometry.
\end{theorem}

\textbf{Proof.} Denote by $k(n)$ the smallest number of distance measurements
that allows us to fix some non-degenerate configuration as above. We need to
show that $k(n)\geq\min(2n-3,\frac{3n}{2}).$ We use induction on $n.$ So
suppose that the statement is true for all sets of fewer than $n$ points, and
that there is a non-degenerate configuration of $n$ points, $A_{1},A_{2}%
,\dots,A_{n}$ that is fixed by $k$ measurements where $k<2n-3$ and
$k<\frac{3n}{2}$. Because $k<\frac{3n}{2}$, there is some point $A_{i},$ which
is being used in less than $3$ measurements. If it is only used in $1$
measurement, the configuration is obviously not fixed. So we can assume that
it is used in two distance measurements, $|A_{i}A_{1}|$ and $|A_{i}A_{2}|.$
Because $A_{1},$ $A_{2},$ and $A_{i}$ are not on the same line, the circle
with radius $|A_{i}A_{1}|$ around $A_{1}$ and the circle with radius
$|A_{i}A_{2}|$ around $A_{2}$ intersect transversally at $A_{i}$. Now remove
$A_{i}$ and these two measurements. By the induction assumption, the remaining
system of $(n-1)$ points admits a small perturbation. This small perturbation
leads to a small perturbation of the original system. Q. E. D.

\bigskip

For $n\leq7$ the above bound coincides with the upper bound $(2n-3)$ for the
minimal number of measurements (see note after Theorem 1). One can show that
for $n\geq8,$ $\frac{3n}{2}$ is the optimal bound, with the regular octagon
being the simplest `distance-exceptional' polygon. In fact there are at least two different ways to define the regular octagon with 12 measurements. 

\begin{example}
Suppose $A_{1}A_{2}A_{3}\dots A_{8}$ is a regular octagon. Its $8$  sides and
the diagonals $|A_{1}A_{5}|,$ $|A_{8}A_{6}|,$  $|A_{2}A_{4}|,$ and, finally,
$|A_{3}A_{7}|$ determine it among all  octagons. A proof is elementary but too messy to
include here, its main idea  is that $|A_{3}A_{7}|$ is the biggest
possible with all other  distances fixed. In fact, one can show that the
distance between the  midpoints of $A_{1}A_{5}$ and $A_{6}A_{8}$ is maximal
when $\angle A_{5}A_{1}A_{8}  = \angle A_{6}A_{5}A_{1} = \frac{3\pi}{8}$.
\end{example}

\begin{example}
Suppose $A_{1}A_{2}A_{3}\dots A_{8}$ is a regular octagon. Its $8$ sides and the $4$ long diagonals determine it among all octagons. This follows from \cite{Connelly82}, Corollary 1 to Theorem 5. In fact, this example can be generalized to any regular $(2n)-$gon: a tensegrity structure with cables for the edges and struts for the long diagonals has a proper stress due to its symmetry and thus is rigid. This provides an optimal bound for the case of even number of points, and one can get the result for the odd number of points just by adding one point at a fixed distance to two vertices of the regular $(2n)-$gon.  We should note that tensegrity networks and their generalizations have been extensively studied, see for example  \cite{ConWhi1}. 
\end{example}

There are many open questions in this area, some of which could be relatively
easy to answer. For example, we do not know whether either of the results for
the cube (9 measurements, or 10 distance measurements) is best possible.
Also, it is relatively easy to determine a regular
hexagon by $7$ measurements, but it is not known if $6$ suffice.

One can also try to extend the general results of this section to the three-dimensional case. Methods of Theorem 3.11 produce a lower estimate of $(2n)$ for the number of distance measurements needed to determine the set of $n$ points, no four of which lie on the same plane. One should also note that many of the exceptional polygons that we have constructed, for instance the square, are unique even when considered in the three-dimensional space: the measurements guarantee their planarity. The question of determining the smallest number of measurements, using distances and angles, with the planarity conditions for the faces, seems to be both the hardest and the  most interesting. But even without the planarity conditions or without the angles the answer is not known.

\section{Acknowledgments}

The authors wish to thank David Allwright for many valuable comments
on previous version of the manuscript.  We also thank Joseph O'Rourke
and the anonymous referee for their exceptionally thorough and
detailed reviews and for pointing out a number of relevant references
to the literature.

\appendix

\section{Quadrilateral-faced hexahedrons with all face-diagonals equal}

\begin{center}
{\large by David Allwright }
\end{center}

In this appendix we classify hexahedrons with quadrilateral faces with all
face-diagonals of length $d$. We shall show that such a hexahedron lies in
either (a) a 1-parameter family with dihedral symmetry of order 6 and all the
faces congruent, or (b) a 2-parameter family with a plane of symmetry and 2
congruent opposite faces joined by 4 symmetric trapezoids. The cube is a
special case of both families and in fact has the maximum volume and the
maximum surface area. We first construct the two families and then show that
the classification is complete.

To construct the first family, let $ABCD$ be a regular tetrahedron with base
$BCD$ in the $xy$-plane, vertex $A$ on the $z$-axis, and edges of length $d$.
Let $l$ be a line parallel to the base and passing through the $z$-axis, and
choose $l$ such that the rotation of the line segment $AB$ through $\pi$ about
$l$ intersects the line segment $CD$. This is a single constraint on a
2-parameter family of lines so there is a 1-parameter family of such lines
$l$. Then let $A^{\prime},B^{\prime},C^{\prime},D^{\prime}$ be the rotations
of $A,B,C,D$ through $\pi$ about $l$. By choice of $l$, the points $A^{\prime
}CB^{\prime}D$ are coplanar and form a convex quadrilateral with both
diagonals of length $d$. But rotation through $\pi$ about $l$ and rotation
through $2\pi/3$ about the $z$-axis together generate a dihedral group of
order 6 permuting these 8 points. Then the images of the quadrilateral
$A^{\prime}CB^{\prime}D$ under that group are the 6 congruent faces of a
hexahedron with all face-diagonals of length $d$.

To construct the second family, let $ABCD$ be a regular tetrahedron with edge
length $d$, and let $P$ be a plane with $AB$ on one side and $CD$ on the
other, and such that the reflection of $AB$ in $P$ intersects $CD$. This is a
single constraint on a 3-parameter family of planes, so there is a 2-parameter
family of such planes $P$. Then let $A^{\prime},B^{\prime},C^{\prime
},D^{\prime}$ be the reflections of $A,B,C,D$ in $P$. By choice of $P$, the
points $A^{\prime}CB^{\prime}D$ form a quadrilateral with both diagonals of
length $d$. Its reflection in $P$ has the same property. The points
$AA^{\prime}C^{\prime}C$ then have $AA^{\prime}$ parallel to $CC^{\prime}$
(both perpendicular to $P$) so they are coplanar and form a symmetric
trapezoid with both diagonals $d$. The same holds for the other faces that cut
$P$, so again we have a hexahedron with the required property.

To complete the classification we now show than \emph{any} hexahedron with all
face-diagonals equal must lie in one of these families. So, suppose we have
such a hexahedron, let $\mathbf{a}_{0}$ be one vertex and let $\mathbf{a}_{1}%
$, $\mathbf{a}_{2}$, $\mathbf{a}_{3}$ be the other ends of the diagonals of
the 3 faces meeting at $\mathbf{a}_{0}$. Then $\mathbf{a}_{1}$, $\mathbf{a}%
_{2}$ and $\mathbf{a}_{3}$ are also face-diagonally opposite one another in
pairs, so $\mathbf{a}_{0}$, $\mathbf{a}_{1}$, $\mathbf{a}_{2}$, $\mathbf{a}%
_{3}$ are the vertices of a regular tetrahedron, $T_{a}$, of edge length $d$
say. In fact, let us choose origin at $\mathbf{a}_{0}$, scaling $d=\sqrt2$ and
orient the coordinate axes so that $\mathbf{a}_{1}=(0,1,1)$, $\mathbf{a}%
_{2}=(1,0,1)$ and $\mathbf{a}_{3}=(1,1,0)$. If we let $\mathbf{b}_{i}$ be the
other vertices of the hexahedron, with $\mathbf{b}_{i}$ diagonally opposite
$\mathbf{a}_{i}$, then the $\mathbf{b}_{i}$ also form a tetrahedron, $T_{b}$,
with edge length $d$, but of the opposite orientation. So in fact we may write
$\mathbf{b}_{i}=\mathbf{b}_{0}-L\mathbf{a}_{i}$ and $L$ will then be a proper
orthogonal matrix. The conditions on the $\mathbf{b}_{0}$ and $L$ now are that
all the faces must be planar, and must not intersect each other except at the
edges where they meet. Each face has vertices $\mathbf{a}_{i},\mathbf{b}%
_{j},\mathbf{a}_{k},\mathbf{b}_{l}$ where $\{i,j,k,l\}$ are some permutation
of $\{0,1,2,3\}$. The condition for face $\mathbf{a}_{0},\mathbf{b}%
_{3},\mathbf{a}_{1},\mathbf{b}_{2}$ to be planar is
\begin{equation}
\label{eq:a1}A_{1}:\qquad0=[\mathbf{a}_{1},\mathbf{b}_{0}-L\mathbf{a}%
_{2},\mathbf{b}_{0}-L\mathbf{a}_{3}] =[\mathbf{a}_{1},L\mathbf{a}%
_{2},L\mathbf{a}_{3}]-[\mathbf{b}_{0},\mathbf{a}_{1},L(\mathbf{a}%
_{2}-\mathbf{a}_{3})].
\end{equation}
Equally, the condition for the face $\mathbf{b}_{0},\mathbf{a}_{2}%
,\mathbf{b}_{1},\mathbf{a}_{3}$ to be planar is
\begin{equation}
\label{eq:b1}B_{1}:\qquad0=[\mathbf{b}_{0}-\mathbf{b}_{1},\mathbf{b}%
_{0}-\mathbf{a}_{2},\mathbf{b}_{0}-\mathbf{a}_{3}] =[L\mathbf{a}%
_{1},\mathbf{a}_{2},\mathbf{a}_{3}]-[\mathbf{b}_{0},L\mathbf{a}_{1}%
,\mathbf{a}_{2}-\mathbf{a}_{3}].
\end{equation}
So the 6 planarity conditions are $A_{1}$, $B_{1}$ and the equations $A_{2}$,
$A_{3}$, $B_{2}$, $B_{3}$ obtained from them by cyclic permutation of the
indices $\{1,2,3\}$. Since we now have 6 inhomogeneous linear equations for
the 3 components of $\mathbf{b}_{0}$, consistency of these equations
constrains $L$. It seems computationally easiest to find what these
constraints are by representing $L$ in terms of a unit quaternion
$q=q_{0}+q_{1}\mathbf{i}+q_{2}\mathbf{j}+q_{3}\mathbf{k}$, so that
$L\mathbf{x}=q\mathbf{x}\overline q$. If $L$ is rotation by $\theta$ about a
unit vector $\mathbf{n}$ then $q=\pm\bigl(\cos({\textstyle{\frac{1}{2}}}%
\theta)+\mathbf{n}\sin({\textstyle{\frac{1}{2}}}\theta)\bigr)$, and we shall
choose the representation with $q_{0}>0$. Then a determinant calculation by
Mathematica shows that consistency of equations $A_{1}$, $A_{2}$, $B_{1}$,
$B_{2}$ gives either $q_{1}=\pm q_{2}$ or some $q_{i}=0$ for $i\ne0$.
Similarly, by cyclic permutation we deduce that either (a) $q_{1}=\pm
q_{2}=\pm q_{3}$, or (b) $q_{1}=0$ or $q_{2}=0$ or $q_{3}=0$. Stating this
geometrically, in case (a) the axis of the rotation $L$ is parallel to one of
the axes of 3-fold rotational symmetry of $T_{a}$; while in case (b) it is
coplanar with one of the pairs of opposite edges of $T_{a}$.

In case (a), if we choose $q_{1}=+q_{2}=+q_{3}$ then the solution is
parametrised by $q_{1}$, $\mathbf{b}_{0}=(b,b,b)$ with $b=(1-4q_{1}%
^{2})/(1-6q_{1}^{2})$, $\mathbf{b}_{1}=\mathbf{b}_{0}-L\mathbf{a}%
_{1}=\mathbf{b}_{0}-\bigl(4q_{1}^{2},(q_{0}-q_{1})^{2},(q_{0}+q_{1}%
)^{2}\bigr)$ and $\mathbf{b}_{2}$, $\mathbf{b}_{3}$ are obtained from
$\mathbf{b}_{1}$ by cyclic permutation of coordinates. This naturally is the
1-parameter family of solutions constructed earlier with 3-fold rotational
symmetry about the $(1,1,1)$ direction. The other choices of $\pm$ signs give
congruent hexahedrons to this family. The solution remains valid for
$-1/\sqrt{12}<q_{1}<+1/\sqrt{12}$: at the ends of this range when $q_{1}%
=\pm1/\sqrt{12}$ the vertex $\mathbf{a}_{i}$ coincides with $\mathbf{b}%
_{i\mp1}$ where those suffices are taken from $\{1,2,3\}\pmod 3$, and for
$|q_{1}|>1/\sqrt{12}$ the polyhedron becomes improper because some pairs of
faces intersect. An example (with $q_{1}=0.2$) is shown in
Figure~\ref{fig:app1a}. \begin{figure}[ptbh]
\centering
\subfigure[]{\includegraphics[width=2.1in]{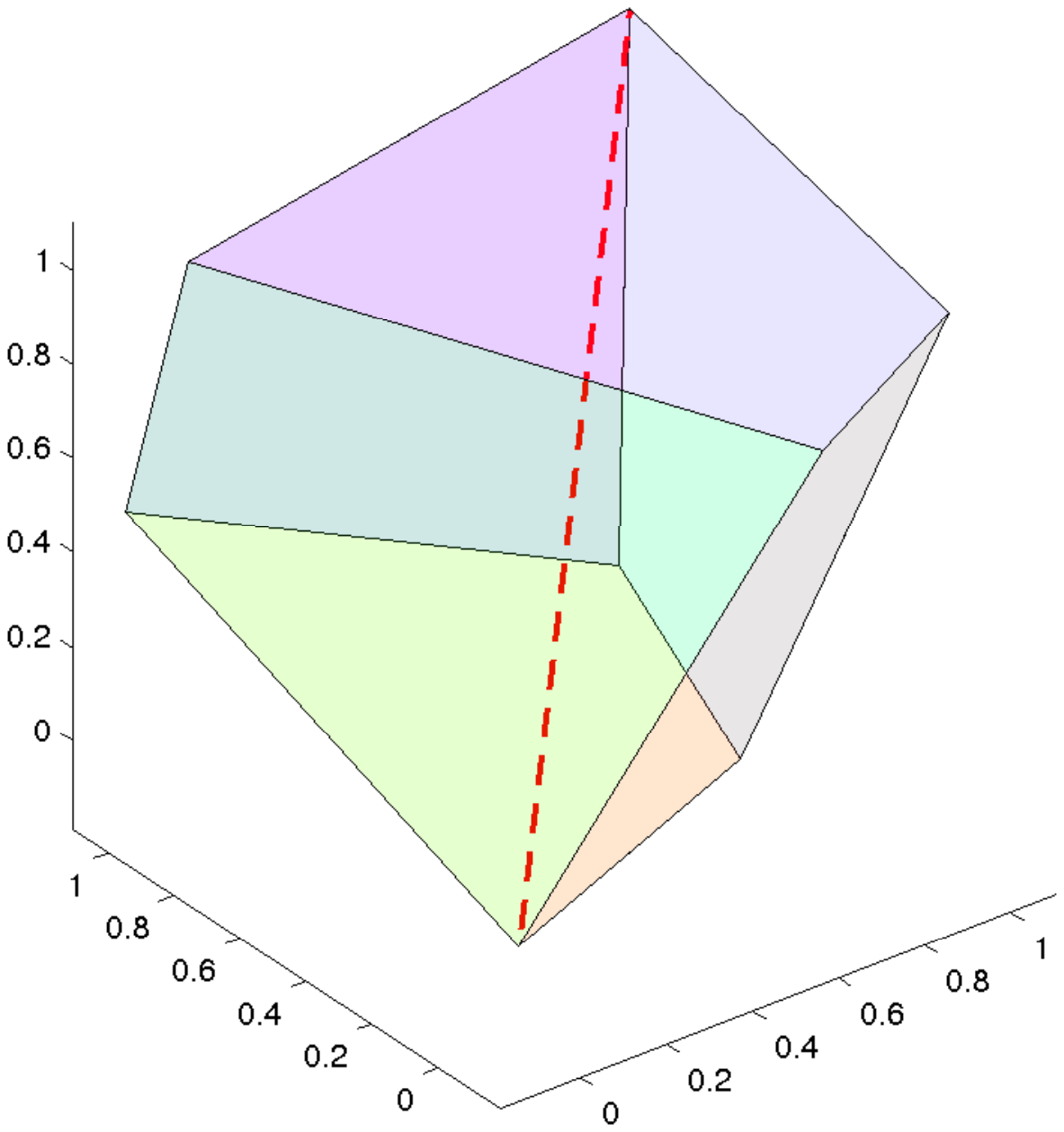}\label{fig:app1a}}
\subfigure[]{\includegraphics[width=2.1in]{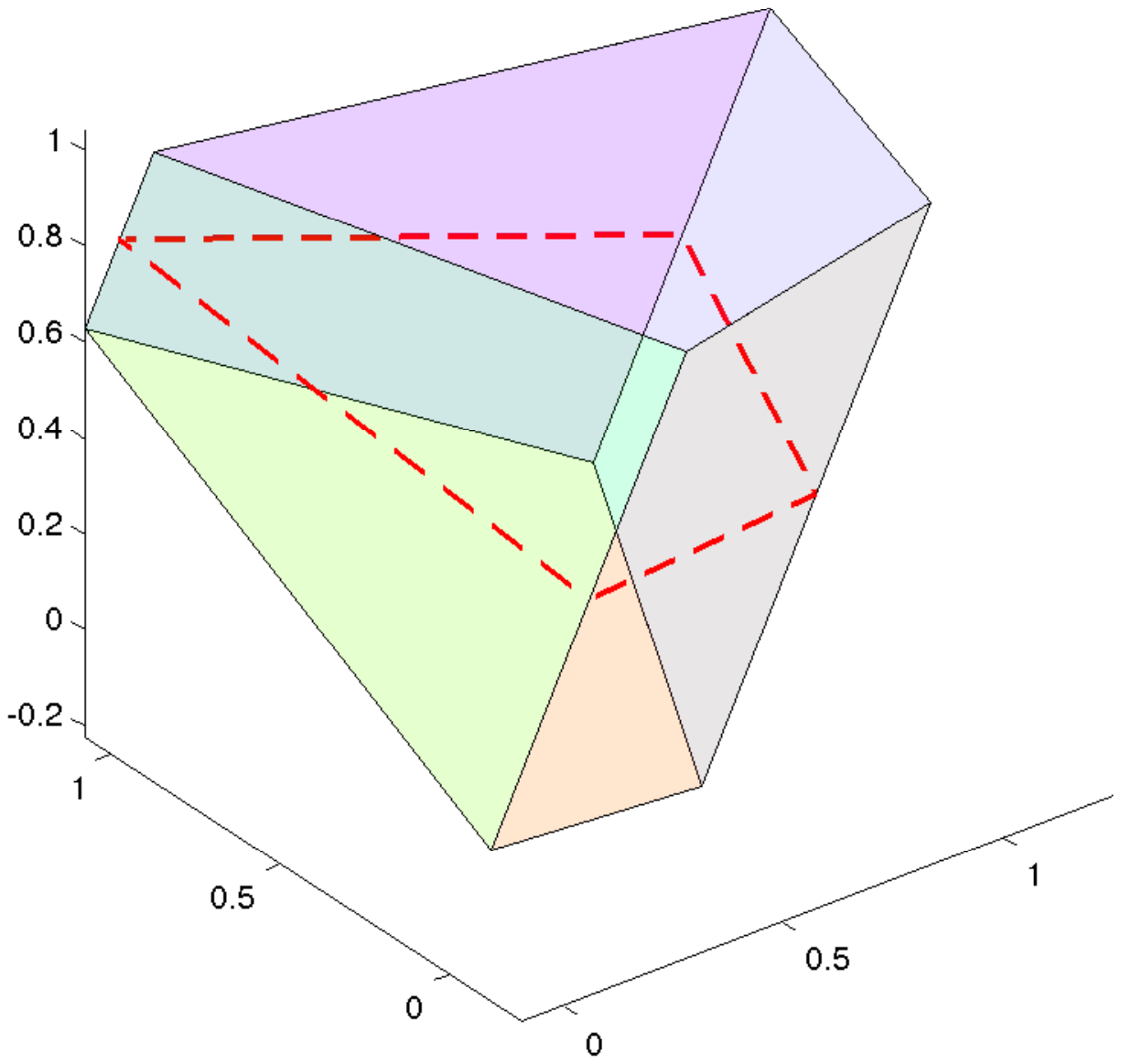}\label{fig:app1b}}
\subfigure[]{\includegraphics[width=2.1in]{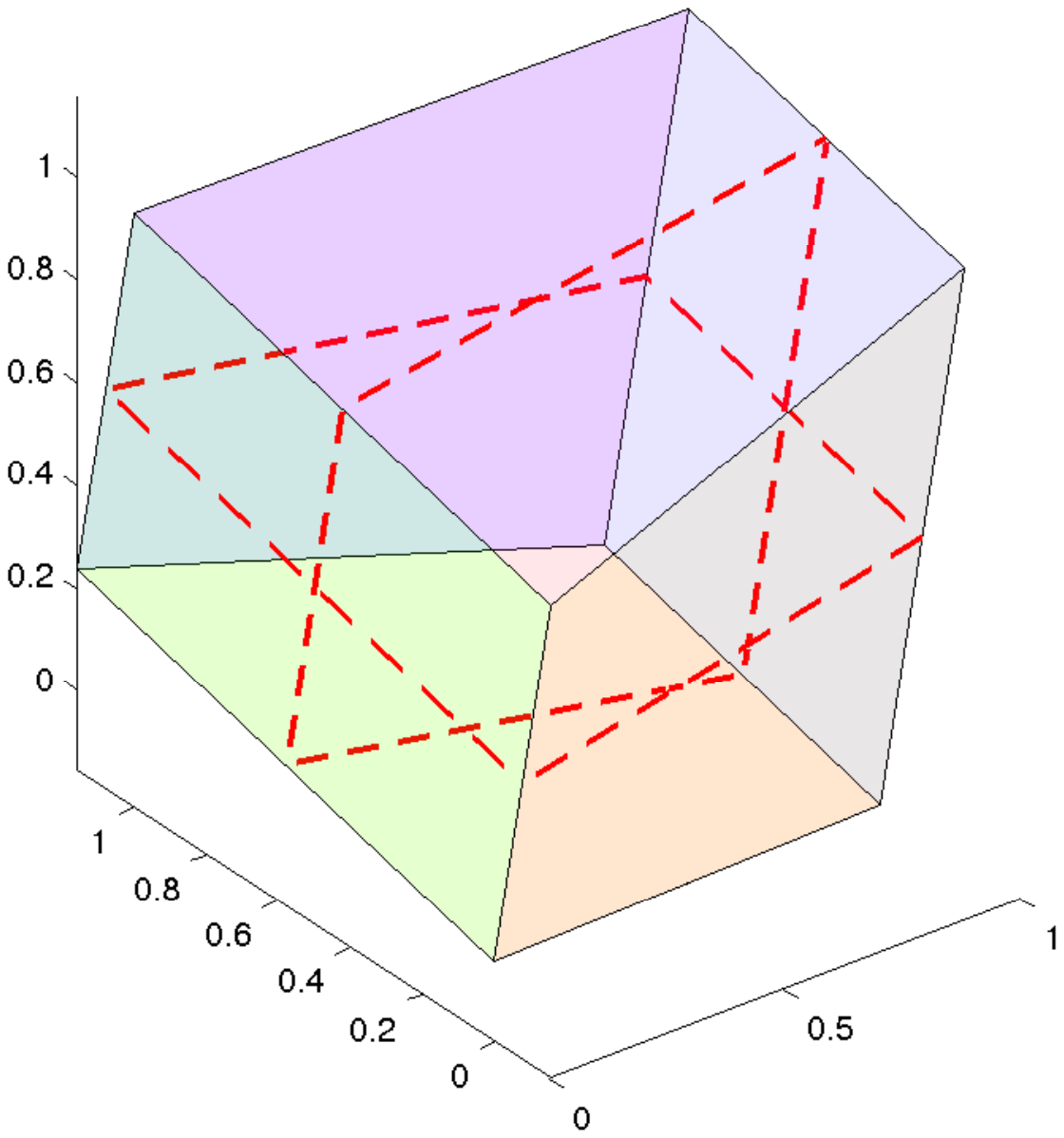}\label{fig:app1c}}  In
case (a) the dashed line is the axis of 3-fold rotational symmetry. In case
(b) the dashed lines mark the plane of symmetry. In case (c) the dashed lines
mark the 2 planes of symmetry. \caption{Hexahedrons with all face-diagonals
equal.}%
\label{fig:app1}%
\end{figure}

In case (b), suppose to be definite we take $q_{3}=0$, so $q_{0}^{2}+q_{1}%
^{2}+q_{2}^{2}=1$. Then the coordinates of $\mathbf{b}_{0}$ are
\begin{equation}
\mathbf{b}_{0}=\left( 1+q_{0}q_{2}+q_{1}q_{2}-q_{2}^{2}+2q_{1}q_{2}^{2}/q_{0},
q_{0}^{2}-q_{0}q_{1}+q_{1}q_{2}+q_{2}^{2}-2q_{1}^{2}q_{2}/q_{0}, q_{0}%
^{2}+q_{0}q_{1}-q_{0}q_{2}+2q_{1}q_{2}\right) .
\end{equation}
This then forms the 2-parameter family constructed earlier, and it can be
checked that the face $\mathbf{a}_{0},\mathbf{b}_{2},\mathbf{a}_{3}%
,\mathbf{b}_{1}$ is congruent to $\mathbf{b}_{0},\mathbf{a}_{1},\mathbf{b}%
_{3},\mathbf{a}_{2}$, and that there is a plane of symmetry $P$ with
$\mathbf{b}_{3-i}$ being the reflection of $\mathbf{a}_{i}$ in $P$. An example
(with $q_{1}=0.2$, $q_{2}=0.25$) is given in Figure~\ref{fig:app1b}, where the
congruent faces are roughly kite-shaped. It remains a valid solution over the
region defined by the inequalities
\begin{equation}
|q_{1}|+|q_{2}|<\frac1{\sqrt2},\qquad\left( 1-(|q_{1}|+|q_{2}|)^{2}\right)
\left( 1-2(|q_{1}|+|q_{2}|)^{2}\right) >2|q_{1}q_{2}|(|q_{1}|+|q_{2}|)^{2}.
\end{equation}
On the edge of this region $P$ passes through a vertex and so two of the faces
fail to be proper quadrilaterals, and outside this region the polyhedron
becomes improper because some pairs of faces intersect.

In case (b) when \emph{two} of the $q_{i}$ vanish, say $q_{2}=q_{3}=0$, then
there are \emph{two} planes of symmetry, and the faces are 2 congruent
rectangles and 4 congruent symmetric trapezoids, illustrated (for $q_{1}=0.2$)
in Figure~\ref{fig:app1c}.

\section{Algorithm for finding a sufficient set of measurements}

In this appendix we give Python code that implements an algorithm to
find a sufficient set of measurements.  The code should be compatible
with versions 2.4 and later of Python.  The code also uses the
\texttt{`numpy'} package, freely available from
\texttt{http://numpy.scipy.org}.

\lstset{language=Python, showspaces=false, showstringspaces=false,
  columns=fullflexible, morekeywords={yield}, basicstyle=\small,
  keywordstyle=\color{blue}, stringstyle=\color{red}, 
  commentstyle={\color{magenta}\itshape}}
\begin{lstlisting}
from numpy import array, reshape, concatenate, transpose, compress, \
    float64, zeros, ones, eye, cross, inner, outer, sometrue
from numpy.linalg import det, svd
from math import sqrt, acos, pi, cos, sin
from itertools import chain, groupby
from operator import itemgetter

def first(iterable):
    return iterable.next()

def all(conditions):
    for c in conditions:
        if not c:
            return c
    return True

def any(conditions):
    for c in conditions:
        if c:
            return c
    return False

#---------------------------------------
# Numerical linear algebra.  We use singular value decomposition to
# determine the nullspace of a matrix.

def norm(v):
    return sqrt(inner(v, v))

def rank(A, eps = 1e-12):
    return sum(svd(A)[1] > 1e-12)

def independent(A, eps = 1e-12):
    """Determine whether the rows of the matrix A are linearly independent"""
    return len(A) == 0 or sum(svd(A)[1] > 1e-12) == len(A)

def nullspace(A, eps = 1e-12):
    """null space of the given matrix A

    Find a return a basis for the space of vectors v such that Av = 0."""

    _, D, VT = svd(A)

    # the second return value (D) of svd is the list of the entries of
    # the diagonal matrix, in decreasing order; if the original matrix
    # has shape (m, n) then D has length min(m, n).  We need to pad it
    # out to length n.

    scale = max(D[0], 1e-3)
    condition_number = min(a for a in D/scale if a > eps)
    if condition_number <= 1e-6:
        # warn that results may be unreliable
        print \
            "Warning: matrix appears ill-conditioned; " + \
            "results may be unreliable." + \
            "Condition number: %s" % condition_number
        print "singular value array: ",D

    return compress(concatenate(
            (D/scale <= eps, [True]*(A.shape[1] - len(D)))), VT, axis=0)

def independent_subset(labelled_rows, A, eps = 1e-12):
    """find an independent subset from a labelled list of vectors

    Given a labelled list of pairs (label, row) where each row is a
    vector of length n, and an initial m-by-n matrix A of rank m,
    return a list of labels corresponding to a maximal subset of the
    given vectors that, together with the rows of A, forms a linearly
    independent set."""

    subset = []
    B = A.copy()
    Bnull = nullspace(B)
    for label, row in labelled_rows:
        if sometrue(abs(inner(Bnull, row)) > eps):
            subset.append(label)
            B = concatenate((B, [row]))
            Bnull = nullspace(B)
    return subset

#---------------------------------------
# Planes, lines and points

# A point in Euclidean 3-space is represented as usual simply as a
# vector of 3 real numbers; it's also occasionally convenient to think
# of it as an element of projective 3-space, identifying [x, y, z]
# with [x : y : z : 1].

# A line in P3 will be represented simply as a pair of points.

# The plane with equation ax + by + cz + d = 0 is represented as a
# quadruple [a, b, c, d] of real numbers; this representation is
# clearly unique only up to scaling by a nonzero constant.  A point
# [x, y, z] lies on a plane [a, b, c, d] if and only if the inner
# product of [x, y, z, 1] with [a, b, c, d] vanishes.

class NoUniqueSolution(Exception):
    pass

def plane_containing_points(points):
    """Find the unique plane containing a given set of points

    Given a list of points in projective 3-space, determine whether they lie
    on a unique plane, and if so return that plane.  If not, raise a
    NoUniqueSolution exception."""

    v = nullspace(array(points))
    if len(v) == 1:
        return v[0]
    elif len(v) > 1:
        raise NoUniqueSolution(
            "more than one plane contains the given points")
    elif len(v) < 1:
        raise NoUniqueSolution("points are not coplanar")

def line_containing_points(points):
    v = nullspace(array(points))
    if len(v) == 2:
        return tuple(v)
    elif len(v) > 2:
        raise NoUniqueSolution("more than one line contains the given points")
    elif len(v) < 2:
        raise NoUniqueSolution("points are not collinear")

def point_avoiding_planes(planes):
    """find a point in P3 not lying on any plane in the given list"""

    n = len(planes)+1
    base = []
    for plane_set in concatenate(
        (transpose([eye(3, dtype=float64)]*n, [1, 0, 2]),
         -transpose([[range(n)]*3], [1, 2, 0])), 2):
        base.append(first(p for p in plane_set
            if all(independent(base + [p, plane]) for plane in planes)))
    return nullspace(array(base))[0]

#---------------------------------------
# class for representing permutations of finite sets; this is used for
# the two permutations of darts required to specify the combinatorial
# structure of a polyhedron.

class Permutation(object):
    """Permutation of a finite set"""

    def __init__(self, pairlist):
        """initialize from a list of key, value pairs"""
        self.perm_map = dict(pairlist)
        self.domain = set(self.perm_map.keys())
        if self.domain != set(self.perm_map.values()):
            raise ValueError("map is not a permutation")

    def __eq__(self, other):
        return self.perm_map == other.perm_map

    def __ne__(self, other):
        return not (self == other)

    def __iter__(self):
        return self.perm_map.iteritems()

    def __getitem__(self, key):
        return self.perm_map[key]

    @staticmethod
    def identity(S):
        """identity permutation on the given set"""
        return Permutation((s, s) for s in S)

    def __mul__(self, other):
        if self.domain == other.domain:
            return Permutation((k, self[other[k]])
                               for k in other.perm_map.keys())
        else:
            raise ValueError("different domains in multiplication")

    def __str__(self):
        return "".join(map(str, self.cycles()[0]))

    __repr__ = __str__

    @staticmethod
    def fromcycle(cycles):
        """initialize from a list of cycles"""
        return Permutation((c[i-1],c[i]) for c in cycles for i in range(len(c)))

    def cycles(self):
        """return a list of the cycles of the given permutation, together with
        a function that maps every element of the domain to the cycle
        containing it."""

        p = dict(self.perm_map)
        def cycle_starting_at(p, v):
            cycle = []
            while v in p:
                cycle.append(v)
                v = p.pop(v)
            return tuple(cycle)
        cycle_list = []
        while p:
            cycle_list.append(cycle_starting_at(p, min(p)))

        cycle_containing = dict((e, c) for c in cycle_list for e in c)
        return cycle_list, (lambda x: cycle_containing[x])

def equivalence_classes(S, pairs):
    """given a set or list S, and a sequence of pairs (s, t) of elements of S,
    compute the equivalence classes of S generated by the relations s ~ t for
    (s, t) in pairs.

    Returns a list of equivalence classes, with each equivalence class
    given as a list of elements of S.

    Example:

    >>> equivalence_classes(range(10), [(1, 4), (2, 6), (2, 3),
                                        (0, 7), (1, 8), (3, 0)])
    [[0, 2, 3, 6, 7], [1, 4, 8], [5], [9]]

    """

    # Algorithm based on that described in `Numerical Recipes in C,
    # 2nd edition', section 8.6; in turn based on Knuth.

    def root(s):
        while s != parent[s]:
            s = parent[s]
        return s

    parent = dict((s, s) for s in S)
    for s, t in pairs:
        s, t = root(s), root(t)
        if s != t:
            parent[max(s, t)] = min(s, t)

    # by this stage, two elements s and t are in the same equivalence
    # class if and only if they have the same root.   So we just
    # have to collect together items with the same root.

    return [map(itemgetter(1), c[1]) for c in
            groupby(sorted((root(s), s) for s in S), itemgetter(0))]

#---------------------------------------

class CombinatorialError(Exception):
    def __init__(self, message, abstract_poly):
        self.message = message
        self.abstract_poly = abstract_poly

    def __str__(self, message, abstract_poly):
        return message

class AbstractPolyhedron(object):
    """Combinatorial data associated to an (oriented) polyhedron.

    The fundamental element of the combinatorial information is the
    `dart'.  One dart is associated to each pair (v, f) where v is a
    vertex, f is a face, and v is incident with f.  Thus each dart has
    an associated vertex and face, each vertex can be described by the
    set of darts around it, and each face can be described by the set
    of darts on it.  Each dart is also naturally associated to an
    edge: in general for a dart d = (v, f) there will be two edges
    incident with both v and f; let w be the next vertex from v on f,
    moving *counterclockwise* around f, then we associate the edge v-w
    to the dart (v, f).

    Now the combinatorial data needed are completely described by two
    permutations: one that associates to each dart d the next dart
    counterclockwise that's associated to the same vertex, and one
    that associates to each dart d the next dart counterclockwise
    that's associated to the same face.

    Attributes of an AbstractPolyhedron instance:

    darts : a list of the underlying darts of the polyhedron;  these can be
        instances of any class that's equipped with an ordering, for example
        integers or strings.
    vertexcycle : a Permutation instance such that for each dart d,
        associated to a vertex v, vertexcycle[d] is the next dart
        around v from d, moving counterclockwise around v.
    facecycle: similar to vertexcycle: gives the permutation of darts
        around each face, again counterclockwise.
    edgecycle: composition of vertexcycle with facecycle.  Sends
        a dart d to the other dart associated to the same edge.

    vertices: list of vertices of the polyhedron.  Each vertex is
        (currently) represented as a list of darts.
    faces: list of faces of the polyhedron
    edges: list of edges of the polyhedron
    components : list of components of the polyhedron; each component
        is represented as a list of darts.  For a connected (nonempty)
        polyhedron, there will be only one component.
    euler : the Euler characteristic of the abstract polyhedron

    V : function that associates to a given dart d the corresponding
        vertex
    F : function associating to each dart d the corresponding face
    E : function associating to each dart d the corresponding edge

    """

    def __init__(self, vertexcycle, facecycle):
        """given permutations `vertexcycle' and `facecycle' describing
        the vertex and face permutations on a set of darts, validate
        and construct the corresponding abstract (oriented) polyhedron.

        """

        if vertexcycle.domain == facecycle.domain:
            self.darts = vertexcycle.domain
        else:
            raise CombinatorialError("the vertex and face permutations "\
                "should have the same underlying set of darts", self)

        self.edgecycle = vertexcycle * facecycle
        self.facecycle = facecycle
        self.vertexcycle = vertexcycle

        if self.edgecycle*self.edgecycle != Permutation.identity(self.darts):
            raise CombinatorialError("every cycle of the edge permutation "\
                "should have length <= 2", self)

        self.vertices, self.V = self.vertexcycle.cycles()
        self.faces, self.F = self.facecycle.cycles()
        self.edges, self.E = self.edgecycle.cycles()
        self.components = equivalence_classes(self.darts,
                chain(self.vertexcycle, self.facecycle))

        self.euler = 2 - len(self.darts) - 2*len(self.components) \
            + len(self.vertices) + len(self.edges) + len(self.faces)

        # additional checks that apply to a connected polyhedron,
        # embeddable in R3 with planar faces.
        self.validate()

    # some basic combinatorial operations;  others can be defined similarly
    def vertices_on_face(self, f): return map(self.V, f)
    def vertices_on_edge(self, e): return map(self.V, e)
    def faces_around_vertex(self, v): return map(self.F, v)
    def faces_around_edge(self, e): return map(self.F, e)
    def edges_on_face(self, f): return map(self.E, f)
    def edges_around_vertex(self, v): return map(self.E, v)

    def __str__(self):
        return """\
Abstract polyhedron with:
  dart set : %s
  vertex permutation : %s
  face permutation : %s
  edge permutation : %s
""" % (self.darts, self.vertexcycle, self.facecycle, self.edgecycle)

    __repr__ = __str__

    def validate(self):
        # each edge should contain at least two darts;  each vertex and face
        # should have at least three.
        bad_edge = any(len(c) < 2 for c in self.edges)
        if bad_edge:
            raise CombinatorialError("edge %s has the wrong length: it "\
                "should involve exactly two darts" % bad_edge, self)

        bad_vertex = any(len(c) < 3 for c in self.vertices)
        if bad_vertex:
            raise CombinatorialError("vertex %s has the wrong length: it "\
                "should involve at least three darts" % bad_vertex, self)

        bad_face = any(len(c) < 3 for c in self.faces)
        if bad_face:
            raise CombinatorialError("face %s has the wrong length: it "\
                "should involve at least three darts" % bad_face, self)

        # verify that the abstract polyhedron is connected
        if len(self.components) != 1:
            raise CombinatorialError("Disconnected polyhedron. " \
                  " Components: %s" % self.components, self)

class Polyhedron(AbstractPolyhedron):

    # a polyhedron should consist of:
    #   (1) an underlying abstract polyhedron
    #   (2) an embedding of that abstract polyhedron

    def __init__(self, vertexcycle, facecycle, embedding, allowScaling = False):

        # initialize the combinatorial data
        AbstractPolyhedron.__init__(self, vertexcycle, facecycle)

        # store position information for vertices (in P3)
        vertexeqn = dict(
            (self.V(d), array(pos + [1.]))
            for d, pos in embedding.items())

        # verify that every vertex has had its position given
        # exactly once
        if not (len(vertexeqn) == len(embedding) == len(self.vertices)):
            raise ValueError("missing vertices or repeated vertex "\
                             "in given embedding")

        # find the coordinates of the plane for each face (in P3)
        # this raises an error if the set of given points is not coplanar
        faceeqn = {}
        for f in self.faces:
            vertex_positions = [vertexeqn[v] for v in self.vertices_on_face(f)]
            try:
                faceeqn[f] = plane_containing_points(vertex_positions)
            except NoUniqueSolution:
                raise ValueError("the points of face %s do not lie "
                                 "in a unique plane" % (str(f)))

        # choose a reference point in P3 that doesn't lie on any of the planes
        # and doesn't lie on the plane at infinity
        basept = point_avoiding_planes(faceeqn.values() + [array([0., 0., 0., 1.])])

        # project into R3, shifting so that the reference point goes
        # to (0, 0, 0) Then none of the faces goes through 0, 0, 0, so
        # each face can be represented by a triple a, b, c,
        # corresponding to the equation ax + by + cz + 1 = 0.
        vertexpos = dict((v, p[:3]/p[3] - basept[:3]/basept[3])
                         for v, p in vertexeqn.items())
        facepos = dict((f, e[:3]*basept[3]/inner(e, basept))
                       for f, e in faceeqn.items())

        # now fill in the relationmatrix (= D(phi))
        ambient_dim = 3*(len(self.vertices)+len(self.faces))
        relationmatrix = zeros((len(self.darts), ambient_dim), dtype = float64)
        for i, d in enumerate(self.darts):
            # construct the row corresponding to a vertex--face pair
            v = self.V(d)
            f = self.F(d)
            vi = 3*self.vertices.index(v)
            fi = 3*(len(self.vertices)+self.faces.index(f))
            relationmatrix[i, vi:vi+3] = facepos[f]
            relationmatrix[i, fi:fi+3] = vertexpos[v]

        # check to see whether the relation matrix has full rank
        if rank(relationmatrix) != len(relationmatrix):
            print ("Warning: relation matrix has rank %d; "
                   "expected rank %d" %
                   (rank(relationmatrix), len(relationmatrix)))

        # compute the vectors of the Lie group: shifts, rotations and scalings
        shifts = [eye(3, dtype=float64) for v in self.vertices] + \
                [outer(facepos[f], facepos[f]) for f in self.faces]
        shifts = reshape(transpose(shifts, [1, 0, 2]), (3, ambient_dim))
        rotations = \
            [cross(eye(3, dtype=float64), vertexpos[v]) for v in self.vertices] + \
            [cross(eye(3, dtype=float64), facepos[f]) for f in self.faces]
        rotations = reshape(transpose(rotations, [1, 0, 2]), (3, ambient_dim))
        if allowScaling:
            scalings = [vertexpos[v] for v in self.vertices] + \
                    [-facepos[f] for f in self.faces]
            scalings = reshape(scalings, (1, ambient_dim))
            liealgebra = concatenate((shifts, rotations, scalings))
        else:
            liealgebra = concatenate((shifts, rotations))

        self.euler_characteristic = \
            len(self.faces) + len(self.vertices) - len(self.edges)

        # Polyhedron attributes:

        # darts : list of darts (= vertex-face pairs)
        # vertices : list of darts representing the vertices, one per vertex
        # faces : list of darts representing the faces, one per face
        # liealgebra contains the vectors of the lie algebra
        # vertexcycle[d] : the dart obtained by moving counterclockwise
        #     about the vertex associated to d
        # facecycle[d] : the dart obtained by moving counterclockwise
        #     about the face associated to d
        # vertexpos[v] : the position of vertex v in R3 (after shifting to
        #     make sure that no faces contain 0, 0, 0)
        # facepos[f] : the plane containing face f (after shifting), given
        #     by a triple (a, b, c) representing the plane with equation
        #     ax + by + cz + 1 = 0.
        # embedding[v] : the original position of vertex v
        # faceeqn[f] : the equation of face f (before shifting), given
        #     by a 4-tuple (a, b, c, d) representing the plane with
        #     equation ax + by + cz + d = 0.

        self.liealgebra = liealgebra
        self.vertexpos = vertexpos
        self.facepos = facepos
        self.allowScaling = allowScaling
        self.ambient_dim = ambient_dim
        self.relationmatrix = relationmatrix
        self.vertexeqn = vertexeqn
        self.faceeqn = faceeqn

    def filter_measurements(self, measurementlist):
        measurements_used = independent_subset([(m, m.derivative())
            for m in measurementlist], self.relationmatrix)

        defect = self.ambient_dim - len(self.liealgebra) - \
            len(self.relationmatrix) - len(measurements_used)
        return measurements_used, defect

    def length(self, v1, v2):
        return LengthMeasurement(self, v1, v2)

    def angle(self, v1, v2, v3):
        return AngleMeasurement(self, v1, v2, v3)

    def dihedral(self, f1, f2):
        return DihedralMeasurement(self, f1, f2)

class Measurement(object):
    """Class representing a particular measurement of a
    polyhedron."""

    # we'll subclass this for length, angle and dihedral angle
    # measurements

    # each subclass should implement value(), derivative(), str()

    pass

# Note that the length measurement actually only needs to depend on
# the abstract polyhedron;  we need a genuine embedded polyhedron only
# to apply the measurement.

class LengthMeasurement(Measurement):
    """Class representing the measurement of the distance
    between two vertices of a polyhedron."""

    def __init__(self, P, v1, v2):
        if not (v1 in P.vertices and v2 in P.vertices):
            raise ValueError("invalid vertices")

        if P.allowScaling:
            raise ValueError("working up to similarity only; "
                             "length measurements not permitted")

        # P is the polyhedron;  v1 and v2 are the relevant vertices
        self.P = P
        self.v1, self.v2 = v1, v2

    def __str__(self):
        return "Length measurement V(%s) -- V(%s)" % (self.v1[0], self.v2[0])

    def value(self):
        return norm(self.P.vertexpos[self.v1] - self.P.vertexpos[self.v2])

    def derivative(self):
        P = self.P
        p1, p2 = P.vertexpos[self.v1], P.vertexpos[self.v2]
        diff = p1 - p2
        length = norm(diff)

        row = zeros((len(P.vertices) + len(P.faces), 3), dtype=float64)
        row[P.vertices.index(self.v1)] += diff/length
        row[P.vertices.index(self.v2)] -= diff/length
        return reshape(row, (P.ambient_dim,))

class AngleMeasurement(Measurement):
    # cosine of the angle v2-v1-v3
    def __init__(self, P, v1, v2, v3):
        if not all(v in P.vertices for v in [v1, v2, v3]):
            raise ValueError("invalid vertices")
        self.P = P
        self.v1, self.v2, self.v3 = v1, v2, v3

    def __str__(self):
        return "Angle measurement V(%s) -- V(%s) -- V(%s)" % \
            (self.v2[0], self.v1[0], self.v3[0])

    def value(self):
        P = self.P
        (p1, p2, p3) = (P.vertexpos[self.v1],
                        P.vertexpos[self.v2], P.vertexpos[self.v3])
        diff2, diff3 = p2 - p1, p3 - p1
        return inner(diff2, diff3)/norm(diff2)/norm(diff3)

    def derivative(self):
        """vector representing the derivative of the cosine of the angle
        measurement between vectors v2-v1 and v3-v1"""

        P = self.P
        (p1, p2, p3) = (P.vertexpos[self.v1],
                        P.vertexpos[self.v2], P.vertexpos[self.v3])
        diff2, diff3 = p2 - p1, p3 - p1
        a23 = inner(diff2, diff3)
        a22 = inner(diff2, diff2)
        a33 = inner(diff3, diff3)
        M = sqrt(a22 * a33)

        cosangle = a23/M

        dv1 = (diff2*a23/a22 + diff3*a23/a33 - diff2 - diff3)/M
        dv2 = (diff3 - diff2*a23/a22)/M
        dv3 = (diff2 - diff3*a23/a33)/M

        row = zeros((len(P.vertices) + len(P.faces), 3), dtype = float64)
        row[P.vertices.index(self.v1)] += dv1
        row[P.vertices.index(self.v2)] += dv2
        row[P.vertices.index(self.v3)] += dv3
        return reshape(row, (P.ambient_dim,))

class DihedralMeasurement(Measurement):
    # measurement of the angle between faces f1 and f2

    def __init__(self, P, f1, f2):
        if not all(f in P.faces for f in [f1, f2]):
            raise ValueError("invalid faces")
        self.P = P
        self.f1, self.f2 = f1, f2

    def __str__(self):
        return "Dihedral measurement F(%s) -- F(%s)" % (self.f1[0], self.f2[0])

    def value(self):
        # angle between two faces:  the faces have corresponding embeddings
        n1, n2 = self.P.facepos[self.f1], self.P.facepos[self.f2]

        # n1 and n2 are (not necessarily unit) normal vectors to the
        # two faces note that n1 and n2 don't necessarily have
        # corresponding orientations: one of n1 and n2 might point
        # outwards while the other points inwards.  We need to fix
        # this.  For now we just compute the (cosine of the) angle
        # between n1 and n2.
        return inner(n1, n2)/norm(n1)/norm(n2)

    def derivative(self):
        P = self.P
        n1, n2 = P.facepos[self.f1], P.facepos[self.f2]

        a11 = inner(n1, n1)
        a12 = inner(n1, n2)
        a22 = inner(n2, n2)

        dn1 = (n2 - n1*a12/a11)/sqrt(a11*a22)
        dn2 = (n1 - n2*a12/a22)/sqrt(a11*a22)

        row = zeros((len(P.vertices) + len(P.faces), 3), dtype=float64)
        row[len(P.vertices) + P.faces.index(self.f1)] += dn1
        row[len(P.vertices) + P.faces.index(self.f2)] += dn2
        return reshape(row, (P.ambient_dim,))

#---------------------------------------
# Output routines

def print_polyhedron(P):
    print "Vertices (number, position, cycle)\n" \
          "----------------------------------"

    print "".join("%2d  " % v[0] +
        "(% 7.4f, % 7.4f, % 7.4f)  " %
                  tuple(P.vertexeqn[v][:3]/P.vertexeqn[v][3]) +
                  str(v) + "\n" for v in P.vertices)

    print "Faces (number, equation, cycle)\n" \
          "-------------------------------"
    print "".join("%2d  " % f[0] +
           "% 7.4fx + % 7.4fy + % 7.4fz = % 7.4f  " % tuple(P.faceeqn[f]) +
           str(f) + "\n" for f in P.faces)

# print results in a meaningful format
def print_results(P, results):
    for r in results:
        print "%s: expected value: % 7.4f" \
            % (r, r.value())

#---------------------------------------
# Platonic solids

tetrahedron = Polyhedron(
    vertexcycle = Permutation.fromcycle([
        [0, 6 ,3], [1, 5, 9], [2, 11, 7], [4, 8, 10]]),
    facecycle = Permutation.fromcycle([range(i, i+3) for i in range(0, 12, 3)]),
    embedding = dict([
        [0, [1., 1., 1.]],
        [1, [1., -1., -1.]],
        [2, [-1., 1., -1.]],
        [4, [-1., -1., 1.]]]))

cube = Polyhedron(
    vertexcycle = Permutation.fromcycle([
        [0, 4, 8], [1, 11, 22], [2, 21, 19], [3, 18, 5],
        [6, 17, 15], [7, 14, 9], [10, 13, 23], [12, 16, 20]]),
    facecycle = Permutation.fromcycle(
        [range(i, i+4) for i in range(0, 24, 4)]),
    embedding = dict([
        [0, [1., 1., 1.]],
        [1, [1., 1., -1.]],
        [2, [1., -1., -1.]],
        [3, [1., -1., 1.]],
        [6, [-1., -1., 1.]],
        [7, [-1., 1., 1.]],
        [10, [-1., 1., -1.]],
        [12, [-1., -1., -1.]]]))

twisted_cube = Polyhedron(
    vertexcycle = Permutation.fromcycle([
        [0, 4, 8], [1, 11, 22], [2, 21, 19], [3, 18, 5],
        [6, 17, 15], [7, 14, 9], [10, 13, 23], [12, 16, 20]]),
    facecycle = Permutation.fromcycle(
        [range(i, i+4) for i in range(0, 24, 4)]),
    embedding = dict([
        [0, [59./51., 451./510., 451./510.]],
        [1, [1., 1., 0.]],
        [2, [13./17., -13./170., -13./170.]],
        [3, [1., 0., 1.]],
        [6, [0., .1, 1.1]],
        [7, [0., 1., 1.]],
        [10, [0., 1.1, 0.1]],
        [12, [0., 0., 0.]]]))

twisted_cube2 = Polyhedron(
    vertexcycle = Permutation.fromcycle([
        [0, 4, 8], [1, 11, 22], [2, 21, 19], [3, 18, 5],
        [6, 17, 15], [7, 14, 9], [10, 13, 23], [12, 16, 20]]),
    facecycle = Permutation.fromcycle(
        [range(i, i+4) for i in range(0, 24, 4)]),
    embedding = dict([
        [0, [-1.1, -1.1, 1.]],
        [1, [1., 1., -1.]],
        [2, [-1.1, .9, -1.]],
        [3, [1., -1., 1.]],
        [6, [.9, .9, 1.]],
        [7, [-1., 1., 1.]],
        [10, [.9, -1.1, -1.]],
        [12, [-1., -1., -1.]]]))

# rotate 2nd tetrahedron by 45 degrees about z-axis
# transformation: [x, y, z] -> [x cost - y sint, x sint + y cost, z]
theta = 0.3
cost = cos(theta)
sint = sin(theta)
twisted_cube3 = Polyhedron(
    vertexcycle = Permutation.fromcycle([
        [0, 4, 8], [1, 11, 22], [2, 21, 19], [3, 18, 5],
        [6, 17, 15], [7, 14, 9], [10, 13, 23], [12, 16, 20]]),
    facecycle = Permutation.fromcycle(
        [range(i, i+4) for i in range(0, 24, 4)]),
    embedding = dict([
        [0, [1., 1., 1.]],
        [2, [1., -1., -1.]],
        [6, [-1., -1., 1.]],
        [10, [-1., 1., -1.]],
        [1, [cost-sint, sint+cost, -1.]],
        [3, [cost+sint, sint-cost, 1.]],
        [7, [-cost-sint, -sint+cost, 1.]],
        [12,[-cost+sint, -sint-cost,-1.]]]))

octahedron = Polyhedron(
    vertexcycle = Permutation.fromcycle(
        [range(i, i+4) for i in range(0, 24, 4)]),
    facecycle = Permutation.fromcycle([
        [0, 4, 8], [1, 11, 22], [2, 21, 19], [3, 18, 5],
        [6, 17, 15], [7, 14, 9], [10, 13, 23], [12, 16, 20]]),
    embedding = dict([
        [0,  [ 1.,  0.,  0.]],
        [4,  [ 0.,  0.,  1.]],
        [8,  [ 0.,  1.,  0.]],
        [12, [-1.,  0.,  0.]],
        [16, [ 0., -1.,  0.]],
        [20, [ 0.,  0., -1.]]]))

t, tr = (sqrt(5)+1.)/2., (sqrt(5)-1.)/2.
dodecahedron = Polyhedron(
    vertexcycle = Permutation.fromcycle([
        [0, 5, 11], [1, 10, 16], [2, 15, 21], [3, 20, 26], [4, 25, 6],
        [7, 29, 58], [8, 57, 54], [9, 53, 12], [13, 52, 49], [14, 48, 17],
        [18, 47, 44], [19, 43, 22], [23, 42, 39], [24, 38, 27], [28, 37, 59],
        [30, 35, 41], [31, 40, 46], [32, 45, 51], [33, 50, 56], [34, 55, 36]]),
    facecycle = Permutation.fromcycle(
        [range(i, i+5) for i in range(0, 60, 5)]),
    embedding = dict([
        [0, [1., 1., 1.]],
        [1, [t, tr, 0.]],
        [2, [1., 1., -1.]],
        [3, [0., t, -tr]],
        [4, [0., t, tr]],
        [7, [-1., 1., 1.]],
        [8, [-tr, 0., t]],
        [9, [tr, 0., t]],
        [13, [1., -1., 1.]],
        [14, [t, -tr, 0.]],
        [18, [1., -1., -1.]],
        [19, [tr, 0., -t]],
        [23, [-tr, 0., -t]],
        [24, [-1., 1., -1.]],
        [28, [-t, tr, 0.]],
        [30, [-1., -1., -1.]],
        [31, [0., -t, -tr]],
        [32, [0., -t, tr]],
        [33, [-1., -1., 1.]],
        [34, [-t, -tr, 0.]]]),
    allowScaling = True)

icosahedron = Polyhedron(
    vertexcycle = Permutation.fromcycle(
        [range(i, i+5) for i in range(0, 60, 5)]),
    facecycle = Permutation.fromcycle([
        [0, 5, 11], [1, 10, 16], [2, 15, 21], [3, 20, 26], [4, 25, 6],
        [7, 29, 58], [8, 57, 54], [9, 53, 12], [13, 52, 49], [14, 48, 17],
        [18, 47, 44], [19, 43, 22], [23, 42, 39], [24, 38, 27], [28, 37, 59],
        [30, 35, 41], [31, 40, 46], [32, 45, 51], [33, 50, 56], [34, 55, 36]]),
    embedding = dict([
        [0, [1., t, 0.]],
        [5, [0., 1., t]],
        [10, [t, 0., 1.]],
        [15, [t, 0., -1.]],
        [20, [0., 1., -t]],
        [25, [-1., t, 0.]],
        [30, [-1., -t, 0.]],
        [35, [-t, 0., -1.]],
        [40, [0., -1., -t]],
        [45, [1., -t, 0.]],
        [50, [0., -1., t]],
        [55, [-t, 0., 1.]]]))


# a tetrahedron with a notch taken out of the center of one edge.  The
# result is a polyhedron with 12 edges, 6 faces and 8 vertices; the
# abstract polyhedron can be realized as a simple polyhedron, but not
# as a convex polyhedron (the corresponding planar graph is not
# 3-connected, so Steinitz' theorem doesn't apply).

notched_tetrahedron = Polyhedron(
    vertexcycle = Permutation.fromcycle([
        [0, 12, 6], [1, 11, 18], [2, 20, 22], [3, 21, 9],
        [4, 8, 17], [5, 16, 13], [7, 14, 15], [10, 23, 19]]),
    facecycle = Permutation.fromcycle([
        [0, 1, 2, 3, 4, 5], [6, 7, 8, 9, 10, 11],
        [12, 13, 14], [15, 16, 17], [18, 19, 20], [21, 22, 23]]),
    embedding = dict([
        [0, [1., 1., 1.]],
        [1, [0.5, 1., 0.5]],
        [2, [-0.5, 0., 0.5]],
        [3, [-0.5, 1., -0.5]],
        [4, [-1., 1., -1.]],
        [5, [-1., -1., 1.]],
        [7, [1., -1., -1.]],
        [10, [0.5, 0., -0.5]]]))

square_pyramid = Polyhedron(
    vertexcycle = Permutation.fromcycle([
        [0, 3, 6, 9], [1, 11, 13], [2, 12, 4], [5, 15, 7], [8, 14, 10]]),
    facecycle = Permutation.fromcycle([
        [0, 1, 2], [3, 4, 5], [6, 7, 8], [9, 10, 11], [12, 13, 14, 15]]),
    embedding = dict([
        [0, [0., 0., 1.]],
        [1, [1., 1., 0.]],
        [2, [-1., 1., 0.]],
        [5, [-1., -1., 0.]],
        [8, [1., -1., 0.]]]))

# square pyramid with two of the triangular faces equal

degenerate_square_pyramid = Polyhedron(
    vertexcycle = Permutation.fromcycle([
        [0, 3, 6, 9], [1, 11, 13], [2, 12, 4], [5, 15, 7], [8, 14, 10]]),
    facecycle = Permutation.fromcycle([
        [0, 1, 2], [3, 4, 5], [6, 7, 8], [9, 10, 11], [12, 13, 14, 15]]),
    embedding = dict([
        [0, [0., 0., 1.]],
        [1, [1., 1., 0.]],
        [2, [0., 0., 0.]],
        [5, [-1., -1., 0.]],
        [8, [1., -1., 0.]]]))

# torus (embedding as polyhedron with 9 quadrilateral faces) defined
# for experimentation purposes only; since it's not a convex
# polyhedron, the assumptions of the paper and the algorithm don't
# apply.  Expect the Polyhedron() call for the torus to throw up a
# warning: "Warning: relation matrix has rank 35; expected rank 36".

s3 = sqrt(3.)
hs3 = s3/2.
torus = Polyhedron(
    facecycle = Permutation.fromcycle(
        [range(i, i+4) for i in range(0, 36, 4)]),
    vertexcycle = Permutation.fromcycle([
        [0, 9, 34, 27],
        [1, 26, 31, 4],
        [2, 7, 16, 13],
        [3, 12, 21, 10],
        [5, 30, 35, 8],
        [6, 11, 20, 17],
        [14, 19, 28, 25],
        [15, 24, 33, 22],
        [18, 23, 32, 29]]),
    embedding = dict([
        [0, [1., 0., 0.]],
        [1, [-.5, hs3, 0.]],
        [2, [-1., s3, 1.]],
        [3, [2., 0., 1.]],
        [5, [-.5, -hs3, 0.]],
        [6, [-1., -s3, 1.]],
        [14, [-1., s3, -1.]],
        [15, [2., 0., -1.]],
        [18, [-1., -s3, -1.]]]))

# szilassi polyhedron;  combinatorial information based on the heawood graph

szilassi = AbstractPolyhedron(
    vertexcycle = Permutation.fromcycle(
        [range(i, i+3) for i in range(0, 42, 3)]),
    facecycle = Permutation.fromcycle(
        [[(i + j) % 42 for i in [0, 5, 31, 27, 26, 40]] for j in range(0, 42, 6)]))

#---------------------------------------
# attempt to find a set of angles that's sufficient for a dodecahedron;
# we can do this using face measurements only

D = dodecahedron
faceangles = []
for f in D.faces:
    vs = D.vertices_on_face(f)
    for i in range(5):
        for j in range(i+1, 5):
            for k in range(j+1, 5):
                faceangles.append(AngleMeasurement(D, vs[i], vs[j], vs[k]))
                faceangles.append(AngleMeasurement(D, vs[j], vs[k], vs[i]))

results, defect = dodecahedron.filter_measurements(faceangles)
print "Regular dodecahedron"
print_polyhedron(dodecahedron)
if defect == 0:
    print "A sufficient set of angle measurements for the dodecahedron"
    print_results(dodecahedron, results)
else:
    print ("The following measurements are independent, "
           "but not infinitesimally sufficient.")
    print "%d more measurements are required" % defect
    print_results(dodecahedron, results)

\end{lstlisting}

\end{document}